\documentclass[a4paper,10pt]{article}
\usepackage[T1]{fontenc}
\usepackage[english]{babel}
\usepackage[normalem]{ulem}
\usepackage[a4paper,top=3cm,bottom=2cm,left=3cm,right=3cm,marginparwidth=1.75cm]{geometry}
\usepackage{multirow} 
\usepackage{amsmath,amssymb,amsfonts,amsthm,esint,dsfont}
\usepackage[boxruled, vlined]{algorithm2e}
\usepackage{algorithmic}
\usepackage{subfig}
\usepackage{float}
\usepackage[colorinlistoftodos]{todonotes}
\usepackage[colorlinks=true, linkcolor =red, citecolor = blue]{hyperref}

\numberwithin{equation}{section}
\newtheorem{theorem}{Theorem}[section]
\newtheorem{proposition}[theorem]{Proposition}

\newtheorem{lemma}[theorem]{Lemma}
\newtheorem{remark}[theorem]{Remark}

\DeclareMathOperator*{\argmin}{arg\,min}
\DeclareMathOperator*{\argmax}{arg\,max}

\def\N{\mathbb{N}}
\def\Z{\mathbb{Z}}
\def\R{\mathbb{R}}
\def\C{\mathbb{C}}

\def\cB{\mathcal{B}}

\def\cD{\mathcal{D}}

\def\cH{\mathcal{H}}

\def\cJ{\mathcal{J}}
\def\cK{\mathcal{K}}

\def\cM{\mathcal{M}}

\def\cP{\mathcal{P}}

\def\cS{\mathcal{S}}
\def\cT{\mathcal{T}}

\def\cV{\mathcal{V}}

\def\ri{{\mathrm{i}}}

\def\per{\mathrm{per}}
\def\Tr{\mathrm{Tr} \,}
\def\rt{\mathrm{t}}

\newif \iffigure
\figuretrue

\title{Numerical reconstruction of the first band(s) in an inverse Hill's problem.}
\date{}
\author{
Athmane Bakhta\thanks{Ecole des Ponts ParisTech, \texttt{athmane.bakhta@cermics.enpc.fr}} \and Virginie Ehrlacher\thanks{Ecole des Ponts ParisTech \& INRIA, \texttt{virginie.ehrlacher@enpc.fr}} \and David Gontier\thanks{Universit\'e Paris-Dauphine, \texttt{gontier@ceremade.dauphine.fr}} \\
\and
}

\begin{document}
\maketitle

\begin{abstract}
This paper concerns an inverse band structure problem for one dimensional periodic Schr\"odinger operators (Hill's operators). 
Our goal is to find a potential for the Hill's operator in order to reproduce as best as possible some given target bands, which may not be realisable. 
We recast the problem as an optimisation problem, and prove that this problem is well-posed when considering singular potentials (Borel measures). We then propose different algorithms to tackle the problem numerically.
\end{abstract}

\section{Introduction} 
\label{sec:intro}
The aim of this article is to present new considerations on an inverse band structure problem for periodic one-dimensional Schr\"odinger operators, also called Hill's operators.
A Hill operator is a self-adjoint, bounded from below operator of the form 
$A^V:= -\frac{d^2}{dx^2} + V$, acting on $L^2(\R)$, and where $V$ is a periodic real-valued potential. Its spectrum is composed of a reunion of intervals, which can be characterised using 
Bloch-Floquet theory as the reunion of the spectra of a  
family of self-adjoint compact resolvent operators $A^V_q$, indexed by an element $q\in \R$ called the {\em quasi-momentum} or {\em k-point} (see~\cite[Chapter XIII]{ReedSimon} and Section~\ref{sec:Bloch_transform}).
The $m^{th}$ band function associated to a periodic potential is the function which maps $q \in \R$ to the $m^{th}$ lowest eigenvalue of $A^V_q$. The properties of these band functions are well-known,
especially in the one-dimensional case (see e.g.~\cite[Chapter XIII]{ReedSimon}).

\medskip

The inverse band structure problem is an interesting mathematical question of practical interest, which can be roughly formulated as follows: {\em 
is it possible to find a potential $V$ so that its first bands are close to some target functions? }

\medskip

A wide mathematical literature answers the question when the target functions are indeed the bands of some Hill's operator, corresponding to some $V_{\rm ref}$.
In this case, we need to {\em recover} a potential $V$ that reproduces the bands of $V_{\rm ref}$. We refer to~\cite{Eskin,EskinRalstonTrubowiz1, EskinRalstonTrubowiz2, PoschelTrubowiz,FreilingYurko,Veliev} for the case when $V_{\rm ref}$ is a regular potential, and to~\cite{hryniv2003inverse,hryniv2004half,hryniv2004inverse2,hryniv2003inverse3,hryniv2006inverse4} when $V_{\rm ref}$ is singular (see also the review~\cite{Kuchment2016}).
The main ideas of the previous references are as follows. First, the band structure of a Hill's operator can be seen as the transformation of an analytic function. In particular, the knowledge of any band 
on an open set is enough to recover {\em theoretically} the whole band structure. A potential is then reconstructed from the high energy asymptotics of the bands. 

\medskip

The previous methods use the knowledge of the behaviour of the high energy bands, and therefore are unsuitable for practical purpose (material design) since we usually have 
no accurate and numerically stable information about these high energy bands. Moreover, in practice, only the low energy bands are usually of interest. 
The fact that there exists no explicit characterisation of the set of the first band functions associated to a given admissible set of periodic potentials is an additional numerical difficulty. 
For applications, it is therefore interesting to know how to construct a potential such that only its first bands are close to some given target functions, which may not be realisable (for instance not analytic).
In this present work, we therefore adopt a different point of view, which, up to our best knowledge, has not been studied: we recast the inverse problem as an optimisation problem.

\medskip

The outline of the paper is as follows. In Section~\ref{sec:main_results}, we recall basic properties about Hill's operators with singular potentials. 
and we state our main result (Theorem~\ref{thm:minimisers}). Its proof is given in Section~\ref{sec:proofs}. Finally, we present in Section~\ref{sec:numerics} some numerical tests and propose an adaptive optimisation algorithm, which is observed to converge faster than the standard one. This adaptive algorithm relies on the use of an a posteriori error estimator for discretised eigenvalue problems, whose computation is detailed in the Appendix.

\section{Spectral decomposition of periodic Schr\"odinger operators, and main results}
\label{sec:main_results}

In this section, we recall some properties of Hill's operators with singular potentials. Elementary notions on the Bloch-Floquet transform~\cite{ReedSimon} are gathered in Section~\ref{sec:Bloch_transform}. The spectral decomposition
of one-dimensional periodic Schr\"odinger operators with singular potentials is detailed in Section~\ref{sec:singular_potentials}, 
building on the results of~\cite{kato1972schrodinger, hryniv2001schr,gesztesy2006spectral,mikhailets2008, Dias2016}. We state our main results in Section~\ref{ssec:mainResults}.

\subsection{Bloch-Floquet transform}\label{sec:Bloch_transform}

We need some notation. Let $\cD'$ denotes the Schwartz space of complex-valued distributions, and let $\cD'_{\rm per} \subset \cD'$ be the space of distributions that are $2\pi$-periodic. In the sequel, 
the unit cell is $\Gamma:=[-\pi, \pi)$, and the reciprocal unit cell (or Brillouin zone) is $\Gamma^* := [-1/2, 1/2]$. For $u\in \cD'_{\rm per}$ and $k\in \Z$, the 
$k^{th}$ normalised Fourier coefficient of $u$ is denoted by $\widehat{u}(k)$. For $s\in \R$, we denote by 
\begin{equation*}
H^s_{\rm per} := \left \{ u \in \cD'_{\rm per}, \ \| u \|_{H^s_\per}^2 :=  \sum_{k\in \Z} (1+|k|^2)^s|\widehat{u}(k)|^2 < +\infty  \right \}
\end{equation*}
the complex-valued periodic Sobolev space, which is a Hilbert space when endowed with its natural inner product. 
We write $H^s_{\per, r}$ for the {\em real-valued} periodic Sobolev space, i.e.
$$
H^s_{\per, r}:= \left\{ u \in H^s_{\rm per}, \quad \forall k \in \Z, \; \widehat{u}(-k) = \overline{\widehat{u}(k)} \right\}.
$$
We also let $L^2_\per := H^{s = 0}_\per$. From our normalisation, it holds that
\begin{equation*}
\forall v,w \in L^2_{\rm per}, \  \langle v,w \rangle_{L^2_{\rm per}} = \int_\Gamma \overline{v}w 
\quad \text{and} \quad
\forall v,w \in H^1_{\rm per}, \  \langle v,w \rangle_{H^1_{\rm per}} = \int_\Gamma \frac{d\overline{v}}{dx} \frac{dw}{dx}+ \int_\Gamma \overline{v}w.
\end{equation*}
%
Lastly, we denote by $C^0_{\rm per}$ the space of $2\pi$-periodic continuous functions, and by $C^\infty_c$ the space of $C^\infty$ functions over $\R$, with compact support.

\medskip

To introduce the Bloch-Floquet transform, we let $\cH := L^2(\Gamma^*, L^2_\per)$. For any element $f \in \cH$, we denote by $f_q(x)$ its value at the point $(q,x)\in \Gamma^* \times \Gamma$. 
The space $\cH$ is an Hilbert space when endowed with its inner product
\begin{equation*}
\forall f,g \in \cH, \quad \langle f, g \rangle_\cH := \int_{\Gamma^*} \int_\Gamma \overline{f_q(x)} g_q(x) dx \ dq.
\end{equation*}
The Bloch-Floquet transform is the map $\cB : L^2(\R) \to \cH$ defined, for smooth functions $\varphi \in C^\infty_c(\R)$, by
\begin{equation*}
\phi_q(x) := \left( \cB \varphi \right)_q(x) := \sum_{R \in \Z} \varphi(x + R) e^{-\ri q (R+x)}.
\end{equation*}
It is an isometry from $L^2(\R)$ to $\cH$, whose inverse is given by
\begin{equation*}
\left( \cB^{-1} \phi \right)(x) := \int_{\Gamma^*} \phi_q(x) e^{iqx}\,dq = \varphi(x).
\end{equation*}
The Bloch theorem states that if $A$ is a self-adjoint operator on $L^2(\R)$ with domain $D(A)$ that commutes with $\Z$-translations, then $\cB A \cB^{-1}$ is diagonal in the $q$-variable. 
More precisely, there exists a unique family of self-adjoint operators $\left( A_q \right)_{q\in\Gamma^*}$ on $L^2_{\rm per}$ such that for all $\varphi \in L^2(\R) \cap D(A)$, 
$$
(A\varphi)(x) = \int_{\Gamma^*} (A_q\phi_q)(x)\,dq.
$$
In this case, we write
\begin{equation*}
A = \int^\oplus_{\Gamma^*} A_q dq.
\end{equation*}

\subsection{Hill's operators with singular potentials}
\label{sec:singular_potentials}

Giving a rigorous mathematical sense to a Hill's operator of the form $-\frac{d^2}{dx^2} + V$ on $L^2(\R)$, when the potential $V$ is singular is not an obvious task. In the present paper, we consider $V \in H^{-1}_{\rm per,r}$, which is a case that was first tackled in~\cite{kato1972schrodinger} 
(see also~\cite{hryniv2001schr, Dias2016,gesztesy2006spectral,mikhailets2008} for recent results).

%
%

The results which are gathered in this section are direct corollaries of results which were proved in these earlier works, particularly in~\cite{hryniv2001schr}.  

\medskip

\begin{proposition}\label{prop:russe}[Theorem~2.1 and Lemma~3.2 of~\cite{hryniv2001schr}]
For all $V \in H^{-1}_{\rm per,r}$, there exists $\sigma_V \in L^2_{\rm per}$ and $\kappa_V \in \R$ such that 
\begin{equation}\label{eq:nudecomp}
V = \sigma_V' + \kappa_V \mbox{ in } \cD'_{\rm per}.
\end{equation}
Moreover, if $a^V: H^1(\R) \times H^1(\R) \to \C$ is the sesquilinear form defined by 
\begin{equation}\label{eq:adecomp}
\forall v,w\in H^1(\R), \quad a^V(v,w)= \int_\R \frac{d \overline{v}}{dx} \; \frac{d w }{dx}+ \int_\R \kappa_V \overline{v} w - \int_\R \sigma_V \left(\frac{d\overline{v}}{dx}w + \overline{v}\frac{dw}{dx}\right),
\end{equation}
then $a^V$ is a symmetric, continuous sesquilinear form on $H^1(\R)\times H^1(\R)$, which 
is closed and bounded from below. Besides, $a^V$ is independent of the choice of $\sigma_V \in L^2_{\rm per}$ and $\kappa_V\in \R$ satisfying~\eqref{eq:nudecomp}.
\end{proposition}

\begin{remark}
The expression~\eqref{eq:adecomp} makes sense whenever $v, w \in H^1(\R)$. This can be easily seen with the Cauchy-Schwarz inequality, and the embedding $H^1(\R) \hookrightarrow L^\infty(\R)$. It is not obvious how to extend this result to higher dimension.
\end{remark}

A direct consequence of Proposition~\ref{prop:russe} is that one can consider the Friedrichs operator on $L^2(\R)$ associated to $a^V$, which is denoted by $A^V$ in the sequel. 
The operator $A^V$ is thus a densely defined, self-adjoint, bounded from below operator on $L^2(\R)$, with form domain $H^1(\R)$ and whose domain is dense in $L^2(\R)$. Formally, it holds that
\[
A^V = - \dfrac{\partial^2}{\partial x^2}  + V.
\]
The spectral properties of the operator $A^V$ can be studied (like in the case of regular potentials) using Bloch-Floquet theory. 

\medskip

The previous result, together with Bloch-Floquet theory, allows to study the operator $A^V$ via its Bloch fibers $\left(A^V_q\right)_{q\in \Gamma^*}$. For $q\in \Gamma^*$, it holds that $A_q^V$ is the self-adjoint extension of the operator
\begin{equation*}
\left|- i \dfrac{d}{dx} + q\right|^2 + V.
\end{equation*}
It holds that $A_q^V$ is a bounded from below self-adjoint operator acting on $L^2_\per$, whose form domain is $H^1_\per$, and with associated quadratic form $a^V_q$, defined by (recall that $H^1_\per$ is an algebra)
\begin{equation} \label{eq:def:aq}
\forall v,w\in H^1_{\rm per}, \quad a_q^V(v,w):= \int_\Gamma\left[ \overline{\left( - i \frac{d}{dx} + q\right)v}  \left(- i \frac{d}{dx}  + q\right) w \right]  + \langle V , \overline{v}w \rangle_{H^{-1}_{\rm per}, H^1_{\rm per}}. 
\end{equation}
In other words, we have 
$$
A^V = \int_{\Gamma^*}^{\oplus} A^V_q\,dq.
$$

\medskip 

The fact that $L^2_{\rm per}$ is compactly embedded in $H^1_{\rm per}$ implies that $A_q^V$ is compact-resolvent. As a consequence, 
there exists a non-decreasing sequence of real eigenvalues $\left( \varepsilon^V_{q,m}\right)_{m\in \N^*}$ going to $+\infty$ and a corresponding orthonormal 
basis $(u^V_{q,m})_{m\in\N^*}$ of $L^2_{\rm per}$ such that 
\begin{equation} \label{eq:vap1}
\forall m\in \N^*, \quad A^V_q u^V_{q,m} = \varepsilon^V_{q,m} u^V_{q,m}.
\end{equation}
The map $\Gamma^* \ni q \mapsto \varepsilon^V_{q,m}$ is called the $m^{th}$ band. 
Since the potential $V$ is real-valued, it holds that $A^V_{-q} = \overline{A^V_q}$, so that $\varepsilon^V_{-q,m} = \varepsilon^V_{q,m}$ for all $q\in \Gamma^*$ and $m\in\N^*$. This implies that
it is enough to study the bands on $[0, 1/2]$. Actually, we have 
\[
\sigma(A^V) = \bigcup_{q \in [0, 1/2]} \bigcup_{m\in \N^*}\{\varepsilon^V_{q,m}\}.
\]

\medskip

In the sequel, we mainly focus on the first band. We write $\varepsilon_q^V := \varepsilon_{q,1}^V$ for the sake of clarity. Thanks to the knowledge of the form domain of $A^V_q$, we know that
\begin{equation} \label{eq:minmax}
\varepsilon^V_q := \min_{\substack{v \in H^1_\per \\ \| v \|_{L^2_\per} = 1}} a^V_q(v,v).
\end{equation}
This characterisation will be the key to our proof. 
When the potential $V$ is smooth (say $V \in L^2_\per$), 
then the map $\Gamma^* \ni q \mapsto \varepsilon_{q,m}^V$ is analytic on $(-1/2, 1/2)$. Besides, it is increasing on $[0, 1/2]$ if $m$ is odd, and decreasing if $m$ is even (see e.g.~\cite[Chapter XIII]{ReedSimon}).

\subsection{Main results}
\label{ssec:mainResults}
The goal of this article is to find a potential $V$ so that the bands of the corresponding Hill's operator are close to some given target functions. In order to do so, we recast the problem as a minimisation one, of the form
\[
V^* \in \argmin_{V \in \cV} \cJ(V).
\]
Unfortunately, we were not able to consider the full setting where the minimisation set $\cV$ is the whole set $H^{-1}_{\per, r}$. The problem was that we were unable to control the negative part of $V$. To bypass this difficulty, we chose to work with potentials that are bounded from below. Such a distribution is necessary a measure (see e.g.~\cite{lieb2001analysis}). 
Hence measure-valued potentials provide a natural setting for band reconstruction. We recall here some basic properties about measures.

\medskip

We denote by $\cM_{\per}^+$ the space of non-negative $2\pi-$periodic regular Borel measures on $\R$, in the sense that for all $\nu \in \cM_\per^+$, and all Borel set $S \in \cB(\R)$, it holds that 
$\nu(S) = \nu(S + 2 \pi) \ge 0$, and $\nu(\Gamma) < \infty$. For all $\epsilon>0$, from the Sobolev embedding $H_{\rm per}^{1/2 + \varepsilon} \hookrightarrow C^0_\per$, 
we deduce that $\cM_\per^+ \hookrightarrow H_{\rm per}^{-1/2-\varepsilon} \hookrightarrow H_{\rm per}^{-1}$, where the last embedding is compact. 
For $\nu \in \cM_\per^+$, we denote by $V_\nu \in H^{-1}_{\per, r}$ the unique corresponding potential, which is defined by duality through the relation:
\begin{equation*}
\forall \phi \in H^1_\per, \quad \int_\Gamma \phi d \nu = \left\langle V_\nu, \phi \right\rangle_{H^{-1}_\per, H^{1}_\per}.
\end{equation*}

\medskip

For $B \in \R$, we define the set of $B$-bounded from below potentials
$$
\mathcal{V}_B:= \left\{ V \in H^{-1}_{\per, r}| \quad \exists \nu \in  \cM_{\rm per}^+,  \quad V = V_\nu  - B \right\} \subset H^{-1}_{\rm per,r}.
$$
This will be our minimisation space for our optimisation problem. Note that $\cV_{B_1} \subset \cV_{B_2}$ for $B_1 \ge B_2$.

\medskip

We now introduce the functional $\cJ$ to minimise. First, we introduce the set $\cT$ of allowed target functions:
\begin{equation}
\label{def:admissible_set}
\mathcal{T}:= \left\{ b \in C^0(\Gamma^*), \quad b \mbox{ is even and } b \mbox{ is increasing on }[0, 1/2]\right\}.  
\end{equation}
Of course, for all $V \in H^{-1}_{\per, r}$, it holds that $\Gamma^* \ni q \mapsto \varepsilon_q^V \in \cT$. Finally, in order to quantify the quality of reconstruction of a band $b\in \mathcal{T}$, we introduce the error functional $\mathcal{J}_{b}: H^{-1}_{\rm per, r} \to \R$ defined by
\begin{equation}\label{eq:defJ}
\forall V \in H^{-1}_{\rm per, r}, \quad \mathcal{J}_{b}(V):= \frac{1}{2}\int_{\Gamma^*} |b(q) - \varepsilon_{q}^V|^2\,dq = \int_0^{1/2} |b(q) - \varepsilon_{q}^V|^2\,dq.
\end{equation}

\medskip

The main result of the present paper is the following. 
\begin{theorem}
\label{thm:minimisers}
Let $b \in \mathcal{T}$, and denote by $b^*:= \fint_{\Gamma^*} b(q)\,dq \in \R$. Then, for all $B > 1/4 -b^*$, there exists a solution $V_{b,B} \in \mathcal{V}_B$ to the minimisation problem
\begin{equation}\label{eq:minpbm}
V_{b,B} \in \argmin_{V\in \mathcal{V}_B} \mathcal{J}_{b}(V).
\end{equation}
\end{theorem}

The proof of Theorem~\ref{thm:minimisers} relies on the following proposition, which is central to our analysis. Both the proofs of Theorem~\ref{thm:minimisers} and Proposition~\ref{lem:central} are provided in the next section.

\begin{proposition}\label{lem:central}
Let $B \in \R$ and let $(V_n)_{n\in\N^*} \subset \mathcal{V}_B$. For all $n\in \N^*$, let $\nu_n \in \mathcal{M}_{\rm per}^+$ such that $V_n := V_{\nu_n} - B$. 
Let us assume that the sequence $\left( \varepsilon^{V_n}_{0} \right)_{n\in\N^*}$ is bounded and such that $\displaystyle \nu_n(\Gamma) \mathop{\longrightarrow}_{n\to +\infty} + \infty$. 
Then, up to a subsequence (still denoted $n$), the functions $q \mapsto \varepsilon_{q}^{V_n}$ converge uniformly to a constant function $\varepsilon \in \R$, with $\varepsilon \ge \frac14 - B$. In other words, there is $\varepsilon \ge \frac14 - B$ such that
\begin{equation} \label{eq:convEps}
\max_{q \in [0, 1/2]} \left|  \varepsilon_{q}^{V_n} - \varepsilon \right| \xrightarrow[n \to \infty]{} 0.	
\end{equation}

Conversely, for all $\varepsilon \geq \frac14 - B$, there is a sequence $\left(V_n\right)_{n\in\N^*} \subset \cV_B$ such that~\eqref{eq:convEps} holds.
\end{proposition}

This result implies that the first band of the sequence of operators $\left(A^{V_n}\right)_{n\in\N^*}$, where $(V_n)_{n\in\N^*}$ satisfies the assumptions of Proposition~\ref{lem:central}, {\em becomes flat}.

\begin{remark} \label{rem:limitNotReachable}
Here we have a sequence of first bands $\left(\varepsilon^{V_n}_q\right)_{n\in\N^*}$ that converges uniformly to a constant function. 
However, as the first band of any Hill's operator must be increasing and analytic, the limit {\em is not} the first band of a Hill's operator.
\end{remark}

\section{Proof of Theorem~\ref{thm:minimisers} and Proposition~\ref{lem:central}}
\label{sec:proofs}

\subsection{Preliminary lemmas}\label{sec:intermediate}

We first prove some intermediate useful lemmas before giving the proof of Proposition~\ref{lem:central} and Theorem~\ref{thm:minimisers}. We start by recording a spectral convergence result.
\begin{proposition}\label{prop:cv}[Theorem~4.1~\cite{hryniv2001schr}]
Let $(V_n)_{n\in\N^*} \subset H^{-1}_{\per,r}$ be a sequence such that $(V_n)_{n\in\N^*}$ converges strongly in $H^{-1}_{\per}$ to some $V\in H^{-1}_{\per,r}$.  Then,
$$
\forall m \in \N^*, \ \max_{ q \in [0, 1/2]} \left| \varepsilon^{V_n}_{q,m} - \varepsilon^V_{q,m} \right| \xrightarrow[n \to \infty]{} 0.
$$
\end{proposition}

In our case, since we are working with potentials that are measures, we deduce the following result.
\begin{proposition}\label{prop:cv_measure}
Let $B \in \R$ and $(V_n)_{n\in\N^*} \subset \cV_B$ be a bounded sequence, in the sense
\[
\sup_{n \in \N} \left\langle V_n, \mathds{1}_{\Gamma} \right\rangle_{H^{-1}_\per, H^1_\per} < \infty.
\]
For all $n\in\N^*$, let $\nu_n \in \mathcal{M}_{\rm per}^+$ such that $V_n = V_{\nu_n} -B$. 
Then, there exists $\nu \in \mathcal{M}_{\rm per}^+$ such that, up to a subsequence (still denoted $n$), $(\nu_n)_{n \in \N}$ converges weakly-* to $\nu$ in $\cM_\per$, and 
$(V_n)_{n\in\N^*}$ converges strongly in $H^{-1}_\per$ to $V := V_\nu -B\in \mathcal{V}_B$. Moreover, it holds that
$$
\forall m \in \N^*, \ \max_{ q \in [0, 1/2]} \left| \varepsilon^{V_n}_{q,m} - \varepsilon^V_{q,m} \right| \xrightarrow[n \to \infty]{} 0.
$$
\end{proposition}

\begin{proof}
The fact that we can extract from the bounded sequence $(\nu_n)_{n\in\N^*}$ a weakly-* convergent sequence in $\mathcal{M}_{\rm per}^+$ is the Prokhorov's theorem applied in the torus $\Gamma^*$. The second part comes from the compact embedding $\cM_\per \hookrightarrow H^{-1}_{\rm per}$. The final part is the direct application of Proposition~\ref{prop:cv}.
\end{proof}

\begin{remark}
This proposition explains our choice to consider measure-valued potentials. Note that a similar result does not hold in the $L^1_\per$ setting for instance.
\end{remark}

We now give a lemma which is standard in the case of regular potentials $V$ (see~\cite{Evans}).

\begin{lemma}\label{lem:pos}
Let $V \in \cV_B$ for some $B \in \R$. The first eigenvector $u_{q=0}^{V} \in H^1_{\rm per}$ of $A^V_{q = 0}$
is unique up to a global phase. It can be chosen real-valued and    positive.
\end{lemma}

\begin{proof}
We use the min-max principle~\eqref{eq:minmax}, and the fact that, for $u \in H^1_\per$, the following holds 
\[
\left| \dfrac{d}{dx} | u | \right| \le \left| \dfrac{d}{dx} u \right| \quad \text{a.e}.
\]
We see that if $u$ is an eigenvector corresponding to the first eigenvalue, then so is $|u|$. We now consider a non-negative eigenvector $u \ge 0$, 
and prove that it is positive. The usual argument is Harnack's inequality. However, it is a priori unclear that it works in our singular setting. To prove it, we write $V = V_{\nu} - B$ for $\nu \in \cM_\per^+$, 
and consider the repartition function $F_\nu$ of $\nu$: $F_\nu(x) := \nu((0,x])$. This function is not periodic, but the function $f_\nu(x) := F_\nu(x) - \nu(\Gamma)\frac{x}{|\Gamma|}$ is. 
Since $F_\nu$ is an non decreasing, right-continuous function, we deduce that $f_\nu \in L^\infty_\per$. Moreover, it holds, in the $H^{-1}_\per$ sense, that $f_\nu' = V_\nu - | \Gamma |^{-1} \nu(\Gamma) = V + B - | \Gamma |^{-1} \nu(\Gamma)$. As a result, we see that $u$ is solution to the minimisation problem
\begin{equation*}
u \in \argmin_{\substack{v \in H^1_{\per, r} \\ \| v \|_{L^2_{\rm per}} = 1}} \left\{  \int_{\Gamma} \left| \dfrac{d v}{dx} \right|^2 + \left( \frac{\nu(\Gamma)}{|\Gamma|} - B \right) - 2 \int_{\Gamma} f_\nu \left( v \dfrac{dv}{dx} \right) \right\}.
\end{equation*}
There exists $\lambda \in \R$ so that the corresponding Euler-Lagrange equations can be written in the weak-form:
\begin{equation*}
{\rm div} \, F (x,u,u') + G(x,u,u') = 0,
\end{equation*}
with
\begin{equation*}
F(x,u,p) = p - f_\nu u \quad \text{and} \quad B(x,u,p) = f_\nu p + \lambda u.
\end{equation*}
We are now in the settings of~\cite[Theorem 1.1]{Trudinger1967harnack}, and we deduce that $u > 0$. The rest of the proof is standard.
\end{proof}

\subsection{Proof of Proposition~\ref{lem:central}}\label{sec:prooflem}

We now prove Proposition~\ref{lem:central}. Let $B \in \R$ and let $V_n = V_{\nu_n} - B \in \cV_B$ with $\nu_n \in \cM_{\rm per}^+$, be a sequence such that the sequence $\left(\varepsilon_{q = 0}^{V_n}\right)_{n\in \N^*}$ 
is bounded and $\nu_n(\Gamma)$ goes to $+ \infty$. Since $\left(\varepsilon^{V_n}_{0}\right)_{n\in\N^*}$ is bounded, then up to a subsequence (still denoted by $n$), 
there exists $\varepsilon\in \R$ such that $\varepsilon^{V_n}_{0}$ converges to $\varepsilon$. Our goal is to prove that the convergence also holds uniformly in $q \in \Gamma^*$.

\medskip

Let $u^{V_n}_{0} \in H^1_{\rm per}$ be the $L^2_{\rm per}$-normalised    positive eigenvector of $A^{V_n}_0$ associated to the eigenvalue $\varepsilon_{0}^{V_n}$ (see Lemma~\ref{lem:pos}). We denote by $\alpha_n := \min_{x\in \Gamma} u_{0}^{V_n}(x) > 0$. Let us first prove that the following convergences hold:
\begin{equation}
\label{eq:convergence_alpha_n}
\alpha_n \int_{\Gamma} u_{0}^{V_n}\,d\nu_n \xrightarrow[n\to +\infty]{} 0
\quad \text{and} \quad
\alpha_n^2 \nu_n(\Gamma) \xrightarrow[n\to +\infty]{} 0.
\end{equation}

From the equality
\begin{equation*}
\int_\Gamma \left|\frac{d}{dx}\left(u_{0}^{V_n}\right)\right|^2 + \int_\Gamma |u_{0}^{V_n}|^2 d\nu_n = \varepsilon^{V_n}_{0} + B,
\end{equation*}
we get
\begin{equation} \label{eq:series_ineq}
\alpha_n^2 \nu_n(\Gamma) \le \alpha_n \int_\Gamma u_{0}^{V_n} d\nu_n \le \int_\Gamma |u_{0}^{V_n}|^2 d\nu_n \le \varepsilon^{V_n}_{0} + B.
\end{equation}
As the right-hand side is bounded, and $\nu_n(\Gamma) \to +\infty$ by hypothesis, this implies $\alpha_n \to 0$. Moreover, we have
$$
0 \leq \int_\Gamma u_{0}^{V_n} \,d\nu_n = a_0^{V_n}(u_{0}^{V_n},\mathds{1}_\Gamma) + B \int_\Gamma u_{0}^{V_n} = (\varepsilon^{V_n}_{0} + B)\int_\Gamma u_{0}^{V_n} \leq  (\varepsilon^{V_n}_{0} + B) |\Gamma|^{1/2},
$$
where we used the Cauchy-Schwarz inequality for the last part. As a result, we deduce that the sequence $\left(\int_\Gamma u_{0}^{V_n} \,d\nu_n\right)_{n\in\N^*}$ is bounded. The first convergence of~\eqref{eq:convergence_alpha_n} follows. The second convergence is a consequence of the first inequality in~\eqref{eq:series_ineq}.

\medskip  

Let $x_n \in \Gamma = [0,2\pi)$ be such that $\alpha_n = u_{0}^{V_n}(x_n)$. The fact that $\alpha_n \to 0$ 
implies that $\displaystyle l_n:= \|u_{0}^{V_n}(x_n + \cdot) - \alpha_n\|^2_{L^2_{\rm per}} \to 1$ and we can thus define for $n$ large enough
$$
v_n := \frac{u_{0}^{V_n}(x_n + \cdot) - \alpha_n}{\|u_{0}^{V_n}(x_n + \cdot) - \alpha_n\|_{L^2_{\rm per}}}.
$$
It holds that $v_n \in H^1_{\rm per}$, $\|v_n\|_{L^2_{\rm per}} = 1$. Besides, it holds that $v_n(0) = 0$. For $q\in \Gamma^*$, we introduce the function $v_{q,n}$ defined by:
\[
\forall x\in \R, \quad v_{q,n}(x) := v_n(x) e^{-iq [ x ]},  \quad \text{where we set} \quad [ x] := x \ {\rm mod} \ 2 \pi.
\]
Thanks to the equality $v_n(0) = 0$, it holds that $v_{q,n} \in H^1_\per$, and that $\|v_{q,n}\|_{L^2_\per} = 1$. This function is therefore a valid test function for our min-max principle\footnote{This construction only works in one dimension. We do not know how to construct similar test functions in higher dimension.}.

\medskip

From the min-max principle~\eqref{eq:minmax} and the expression~\eqref{eq:def:aq}, we obtain
\begin{align*}
B + \varepsilon_{q}^{V_n} &  \leq 
B + a^{V_n}_q(v_{q,n}, v_{q,n}) \\
& = \int_\Gamma \left|\left( - i \frac{d}{dx} + q \right) v_{q,n} \right|^2 + \int_\Gamma |v_{q,n}|^2 \,d\nu_n 
= \int_\Gamma \left|\frac{dv_n}{dx}\right|^2 + \int_\Gamma |v_n|^2 \,d\nu_n  \\
& = \frac{1}{l_n} \left(\int_\Gamma \left|\frac{d}{dx}\left(u_{0}^{V_n}(x_n + \cdot)\right)\right|^2 + \int_\Gamma |u_{0}^{V_n}(x_n + \cdot) - \alpha_n|^2 \,d\nu_n\right)\\
& = \frac{1}{l_n}\left(\int_\Gamma \left|\frac{d}{dx}\left(u_{0}^{V_n}\right) \right|^2 + \int_\Gamma |u_{0}^{V_n}|^2\,d\nu_n - 2\alpha_n\int_\Gamma u_{0}^{V_n}\,d\nu_n + \alpha_n^2 \nu_n(\Gamma)\right) \\
& = \frac{1}{l_n}\left(B + \varepsilon_{0}^{V_n} - 2\alpha_n\int_\Gamma u_{0}^{V_n}\,d\nu_n + \alpha_n^2 \nu_n(\Gamma)\right).
\end{align*}
We infer from these inequalities, and from (\ref{eq:convergence_alpha_n}) that 
$$
0\leq \max_{q\in \Gamma^*} \left| \varepsilon_{q}^{V_n} - \varepsilon_{0}^{V_n} \right| \leq  \left( B + \varepsilon_{0}^{V_n}\right) \left(\frac{1}{l_n} -1\right) + \frac{1}{l_n}\left(- 2\alpha_n\int_\Gamma u_{0}^{V_n}\,d\nu_n + \alpha_n^2 \nu_n(\Gamma)\right) \xrightarrow[n\to +\infty]{} 0.
$$
This already proves the convergence~\eqref{eq:convEps}.

\medskip

To see that $\varepsilon \ge \frac14 - B$, we write, for $V = V_\nu - B$ with $\nu \in \cM_\per^+$ that
\[
\forall q \in [-1/2, 1/2], \quad	A^V_q = \left| - i \dfrac{d }{dx} + q \right|^2  + V_{\nu} - B \ge \left| - i \dfrac{d }{dx} + q \right|^2 - B \ge q^2 - B,
\]
where we used the fact that the lowest eigenvalue of $\left| - i \dfrac{d }{dx} + q \right|^2$ is $q^2$ for $q \in [-1/2, 1/2]$ (this can be seen with the Fourier representation of the operator). 
As a consequence, for $q = \frac12$, we obtain that for all $V \in \cV_B$, $\varepsilon^V_{q = 1/2} \ge \frac14 - B$. The result follows.

\medskip

To prove the converse, we exhibit an explicit sequence of measures $(\nu_n)_{n\in\N^*} \subset \cM_\per^+$ such that $\varepsilon_{q}^{V_{\nu_n}} \to \frac14$. 
The general result will follow by taking sequences of the form $V_n  = V_{\nu_n} + \left( \varepsilon - \frac14\right) - B$. We denote by $\delta_x$ the Dirac mass at $x \in \R$, and consider, for $\lambda > 0$, the measure
\begin{equation} \label{eq:DiracComb}
\nu_\lambda := \lambda \sum_{k\in \Z} \delta_{2\pi k} \in \mathcal{M}^+_{\rm per}.
\end{equation}
From the first part of the Proposition, it is enough to check the convergence for $q = 0$. We are looking for a solution to (we denote by $\omega_\lambda^2 := \varepsilon_{0}^{V_{\nu_\lambda}} \geq 0$ for simplicity)
\begin{equation} \label{eq:withDirac}
- u'' + \lambda \delta_0 u(0) = \omega_\lambda^2 u, \quad u \ge 0, \quad u(2 \pi) = u(0).
\end{equation}
On $(0, 2\pi)$, $u$ satisfies the elliptic equation $-u'' = \omega_\lambda^2 u$, hence is of the form
\[
u(x) = C e^{i \omega_\lambda x} + D e^{-i\omega_\lambda x},
\]
for some $C,D\in \R$. The continuity of $u$ at $2 \pi$ implies $C e^{2 i \pi \omega_\lambda} + D e^{- 2 i \pi \omega_\lambda} = C + D$. 
Moreover, integrating~\eqref{eq:withDirac} between $0^-$ and $0^+$ leads to the jump of the derivative $- u'(0) + u'(2 \pi) + \lambda u(0)=  0$, or
\[
i \omega_\lambda \left( D - C \right) + i \omega_\lambda \left( C e^{2 i \pi \omega_\lambda} - D e^{-2 i \pi \omega_\lambda} \right) + \lambda (C + D) = 0.
\]
We deduce that $(C,D)$ is solution to the $2 \times 2$ matrix equation
\begin{equation*}
\begin{pmatrix}
1 - e^{2 i \pi \omega_\lambda} & 1 - e^{-2 i \pi \omega_\lambda} \\
- i \omega_\lambda \left( 1 - e^{2 i \pi \omega_\lambda} \right)  + \lambda &  i \omega_\lambda \left( 1 - e^{- 2 i \pi \omega_\lambda } \right)  + \lambda
\end{pmatrix}
\begin{pmatrix}
C\\ 
D
\end{pmatrix} 
= 
\begin{pmatrix}
0 \\ 0
\end{pmatrix} .
\end{equation*}
The determinant of the matrix must therefore vanish, which leads to
\begin{equation*}
1 = \cos(2 \pi \omega_\lambda) + \frac{\lambda}{2} \frac{\sin (2 \pi \omega_\lambda)}{\omega_\lambda}.
\end{equation*}
As $\lambda \to \infty$, one must have $\omega_\lambda \to 1/2$, or equivalently $\varepsilon_{0}^{V_{\nu_\lambda}} \to 1/4$. The result follows.

\subsection{Proof of Theorem~\ref{thm:minimisers}}\label{sec:proofth}

We are now in position to give the proof of Theorem~\ref{thm:minimisers}. Let $b\in \mathcal{T}$ and $B > 1/4 - b^*$ where $b^*:= \fint_{\Gamma^*} b(q)\,dq$. Let $V_n = V_{\nu_n} - B \subset \mathcal{V}_B$ be a minimising sequence associated to problem~\eqref{eq:minpbm}. 

\medskip

Let us first assume by contradiction that $\nu_n(\Gamma) \to \infty$. Then, according to Proposition~\ref{lem:central}, up to a subsequence (still denoted by $n$), there exists $\varepsilon \ge \frac14 - B$ such that $\varepsilon^{V_n}_q$ converges uniformly in $q \in \Gamma^*$ to the constant function $\varepsilon$. Also, from the second part of Proposition~\ref{lem:central}, the fact that $B > \frac14 - b^*$ and the fact that $b^*$ is the unique minimiser to
\begin{equation}\label{eq:minmean}
\inf_{c \in \R} \cK_b(c),
\end{equation}
where $\cK_b(c) := \int_{[0,1/2]} |b(q) - c|^2\,dq$ for all $c\in \R$, it must hold that $\varepsilon = b^*$.

\medskip

We now prove that 
$$
\mathop{\inf}_{V\in \mathcal{V}_B} \mathcal{J}_b(V) \neq \inf_{c \in \R} \cK_b(c) = \cK_b(b^*).
$$
To this aim, we exhibit a potential $W\in \mathcal{V}_B$ such that $\mathcal{J}_b(W) < \mathcal{K}_b(b^*)$. Since $b$ is continuous and    increasing on $[0, 1/2]$, 
there exists a unique $q^* \in (0, 1/2)$ such that $b(q^*) = b^*$. We choose $\delta >0$ small enough such that $ 0 < q^* - \delta < q^* + \delta < 1/2$, and set 
$$
\eta^{\rm ext}:= \int_0^{q^* - \delta} |b(q)-b^*|^2\,dq + \int_{q^*+\delta}^{1/2} |b(q) - b^*|^2\,dq 
\quad \text{and} \quad
\eta^{\rm int}:= \int_{q^* - \delta}^{q^* + \delta} | b(q) -b^*|^2\,dq,
$$
so that $\cK_b(b^*) = \eta^{\rm ext} + \eta^{\rm int}$. Since $b$ is    increasing and continuous, it holds that $\eta^{\rm int}>0$ and $\eta^{\rm ext} >0$, and that $b(q^*-\delta) < b^* < b(q^*+\delta)$.

\medskip

We now choose a constant $\sigma>0$ such that
\[
0 < \sigma <\min \left\{  \frac{\eta^{\rm int}}{8 \delta}, B + b^* - \frac14, b^*- b(q^*-\delta), b(q^*+\delta) -b^* \right\}.
\]
Let $\nu_n$ be the measure defined in~\eqref{eq:DiracComb} for $\lambda = n \in \N$, and let
\[
\widetilde{W}_n := V_{\nu_n} + b^* - \frac14.
\]
Since $\varepsilon_q^{\widetilde{W}_n}$ converges to $b^*$ uniformly in $\Gamma^*$, there exists $n_0 \in \N^*$ large enough such that 
$$
\forall q\in \Gamma^*, \quad \left|\varepsilon_{q}^{\widetilde{W}_{n_0}} - b^*\right| < \sigma/2.
$$
We then define
\[
W := \widetilde{W}_{n_0} + b^* - \varepsilon_{q^*}^{\widetilde{W}_{n_0}} = V_{\nu_n} + \left[ \left( B + b^*- \frac14 \right)  - \left( \varepsilon_{q^*}^{\widetilde{W}_{n_0}} - b^* \right) \right] - B.
\]
Since $\sigma < B + b^* - 1/4$, it holds that $W \in \cV_B$. Moreover, it holds that $b^* - \sigma < \varepsilon_{q}^W  < b^* + \sigma$ for all $q \in \Gamma^*$. Finally, for $q = q^*$, we have $\varepsilon^W_{q^*} = b^*$.

\medskip

Let us evaluate $\mathcal{J}_b(W)$. We get
$$
\mathcal{J}_b(W) = \int_0^{q^* - \delta} |b(q) - \varepsilon_{q}^W|^2\,dq +  \int_{q^* - \delta}^{q^* + \delta} |b(q) - \varepsilon_{q}^W|^2\,dq + \int_{q^* + \delta}^{1/2} |b(q) - \varepsilon_{q}^W|^2\,dq.
$$
For the first part, we notice that for $0 \leq q < q^* - \delta$, we have 
\[
b(q) < b(q^* - \delta) < b^* - \sigma <  \varepsilon_{q}^W <  \varepsilon_{q^*}^W = b^*.
\]
This yields that
$$
\forall \ 0 \leq q < q^* - \delta, \quad |b(q) - \varepsilon_{q}^W| = \varepsilon_{q}^W - b(q) < b^* - b(q) = |b(q) - b^*|.
$$
Integrating this inequality leads to
$$
\int_0^{q^* - \delta} |b(q) - \varepsilon_{q}^W|^2\,dq  < \int_0^{q^* - \delta} |b(q) - b^*|^2\,dq.
$$
Similarly, we obtain that
$$
\int_{q^* + \delta}^{1/2} |b(q) - \varepsilon_{q}^W|^2\,dq  < \int_{q^* + \delta}^{1/2} |b(q) - b^*|^2\,dq.
$$
Lastly, for the middle part, we have
$$
\int_{q^* - \delta}^{q^* + \delta} |b(q) - \varepsilon_{q}^W|^2\,dq  < 2 \delta \left[ \varepsilon_{q^*+\delta}^W - \varepsilon_{q^*-\delta}^W \right] \leq  4\delta \sigma \leq \frac{\eta^{\rm int}}{2} <  \int_{q^* - \delta}^{q^* + \delta} |b(q) - b^*|^2\,dq.
$$
Combining all these inequalities yields that $\mathcal{J}_b(W) < \mathcal{K}_b(b^*)$. This contradicts the minimising character of the sequence $(V_n)_{n\in \N^*}$.

\medskip

Hence the sequence $\left(\nu_n(\Gamma)\right)_{n\in\N^*}$ is bounded. The proof of Theorem~\ref{thm:minimisers} then follows from Proposition~\ref{prop:cv_measure}.

\section{Numerical tests} \label{sec:numerics} 
In this section, we present some numerical results obtained on different toy inverse band structure problems. We propose an adaptive
optimisation algorithm in which the different discretisation parameters are progressively increased. Such an approach, although heuristic, shows a significant gain 
in computational time on the presented test cases in comparison to a naive optimisation approach. 

\medskip

In Section~\ref{sec:discretisation}, we present the discretised version of the inverse band problem for multiple target bands. We present the different optimisation procedures used for this problem (direct 
and adaptive) in Section~\ref{sec:optimisation}. Numerical results on different test cases are given in Section~\ref{sec:numerical_results}. 
The reader should keep in mind that although the proof given in the previous section only works for the reconstruction of the first band, it is possible to numerically look for methods that reproduce several bands.

\subsection{Discretised inverse band structure problem}\label{sec:discretisation}
For $k\in \Z$, we let $e_k(x):= \frac{1}{\sqrt{2\pi}}e^{ikx}$ be the $k$-th Fourier mode. For $s\in \N^*$, we define by
\begin{equation}
\label{def:X_s}
X_s:= \mbox{\rm Span} \left\{ e_k, \ k \in \Z, \ |k| \leq s \right\}
\end{equation}
the finite dimensional space of $L^2_\per$ consisting of the $N_s := 2s+1$ lowest Fourier modes. We denote by $\Pi_{X_s}: L^2_{\rm per} \to X_s$ the $L^2_{\rm per}$ orthogonal projector onto $X_s$. In practice, the solutions of the eigenvalue problem (\ref{eq:vap1}) are approximated using a Galerkin method in $X_s$. We denote by $\varepsilon^{V,s}_{q,1} \leq \cdots \leq \varepsilon^{V,s}_{q,N_s}$ the eigenvalues (ranked in increasing order, counting multiplicity) of the operator $A_q^{V,s}:= \Pi_{X_s} A_q^V\Pi_{X_s}^*$. We also 
denote by $(u_{q,1}^{V,s}, \cdots, u_{q,N_s}^{V,s})$ an orthonormal basis of $X_s$ composed of eigenvectors associated to these eigenvalues so that
\begin{equation}\label{eq:eigdisc}
\forall 1\leq j \leq N_s, \quad A_q^{V,s}u_{q,j}^{V,s} = \varepsilon^{V,s}_{q,j}u_{q,j}^{V,s}.
\end{equation}
An equivalent variational formulation of (\ref{eq:eigdisc}) is the following: 
$$
\forall 1\leq j \leq N_s, \quad \forall v \in X_s, \quad a_q^V\left(u_{q,j}^{V,s},v\right) = \varepsilon^{V,s}_{q,j} \left\langle u_{q,j}^{V,s}, v \right\rangle_{L^2_{\rm per}}.
$$
As $s$ goes to $+ \infty$, it holds that $\displaystyle \varepsilon^{V,s}_{q,m} \mathop{\longrightarrow}_{s\to +\infty}  \varepsilon^{V}_{q,m}$. 

\medskip

In order to perform the integration in~\eqref{eq:defJ}, we discretise the Brillouin zone. We use a regular grid of size $Q \in \N^*$, and set
\[
\Gamma^*_Q:= \left\{ -\frac12 + \frac{j}{Q}   , \ j \in \{ 0, \cdots, Q-1\} \right\}.
\]
We emphasise that since the maps $q \mapsto \varepsilon_{q,m}$ are analytic and periodic, the discretisation error coming from the integration will be exponentially small with respect to $Q$. In practice, we fix $Q \in \N^*$.

\medskip

Let $M\in\N^*$ be a desired number of targeted bands and $b_1,\cdots, b_M \in C^0_\per$ be real-valued even functions, and such that $b_m$ is increasing when $m$ is odd and decreasing when $m$ is even. Our cost functional is therefore $\cJ: H^{-1}_{\per, r} \to \R$, defined by
$$
\forall V\in H^{-1}_{\rm per,r}, \quad \cJ(V) :=   \frac{1}{Q}\sum_{q\in \Gamma^*_Q} \sum_{m=1}^M |b_m(q) - \varepsilon_{q,m}^{V}|^2.
$$
Its discretised version, when the eigenvalues problems are solved with a Galerkin approximation, is
$$
\forall s\in \N^*, \quad \forall V\in H^{-1}_{\rm per,r}, \quad \cJ^s(V) :=   \frac{1}{Q}\sum_{q\in \Gamma^*_Q} \sum_{m=1}^M |b_m(q) - \varepsilon_{q,m}^{V,s}|^2.
$$

Recall that our goal is to find a potential $V \in H^{-1}_{\rm per,r}$ which minimise the functional $\cJ^s$. In practice, 
an element $V\in H^{-1}_{\rm per,r}$ is approximated with a finite set of Fourier modes. For $p\in \N^*$, we denote by
\begin{equation}
\label{def:Y_p}
Y_p:= {\rm Span}\left\{ \sum_{k\in \Z, \; |k|^2\leq p} \widehat{V}_k e_k, \ \forall k \in \Z, \ |k| \leq p, \ \overline{\widehat{V}_{-k}} = \widehat{V}_k \right\}.
\end{equation}
Altogether, we want to solve
\[
V^{s,p} := \argmin_{V \in Y_p} \cJ^s(V).
\]

\subsection{Algorithms for optimisation procedures}\label{sec:optimisation}

\subsubsection{Naive algorithm}

We first present a naive optimisation procedure, using a gradient descent method, where the parameters $s$ and $p$ are fixed beforehand. We tested three different versions of the gradient descent algorithm: 
steepest descent (\textbf{SD}), conjugate gradient with Polak Ribiere formula (\textbf{PR}) and quasi Newton with the Broyden-Fletcher-Goldfarb-Shanno formula (\textbf{BFGS}). We do not detail here these 
classical descents and corresponding line search routines for the sake of conciseness and refer the reader to~\cite{BakhtaPhD,Bonnans}. 

\medskip

For all $V \in H^{-1}_{\rm per,r}$, there exists real-valued coefficients $\left(c_k^V\right)_{k\in\N}$ and $\left(d_k^V\right)_{k\in \N^*}$ such that 
$$
V(x) = c^V_0 + \sum_{k\in \N^*} c^V_k \cos(kx) + d_k^V \sin(kx), \quad \mbox{ and } \sum_{k\in\N^*} (1 + |k|^2)^{-1} \left( |c_k^V|^2 + |d_k^V|^2 \right) < +\infty.
$$
For all $k\in \N$ (respectively $k\in \N^*$), we can express the derivative $\partial_{c_k^V} \cJ^s(V)$ (respectively $\partial_{d_k^V} \cJ^s(V)$) exactly in terms of the Bloch eigenvectors $u_{q,m}^{V,s}$. Indeed, it holds that
\[
\partial_{c_k^V} \cJ^s(V) = \frac{1}{Q}\sum_{q\in \Gamma^*_Q} \sum_{m=1}^M 2 \left( \varepsilon_{q,m}^{V,s} - b_m(q) \right) \partial_{c_k^V} \left( \varepsilon_{q,m}^{V,s} \right).
\]
On the other hand, from the Hellman-Feynman theorem, it holds that
\[
\partial_{c_k^V} \left( \varepsilon_{q,m}^{V,s} \right) = \left\langle u_{q,m}^{V,s}, \partial_{c_k^V} A^V_q, u_{q,m}^{V,s} \right\rangle = \langle u_{q,m}^{V,s}, \cos(k\cdot) u_{q,m}^{V,s}\rangle_{L^2_{\rm per}}.
\]
Similarly, for all $k\in \N^*$, 
\[
\partial_{d_k^V} \left( \varepsilon_{q,m}^{V,s} \right) = \left\langle u_{q,m}^{V,s}, \partial_{d_k^V} A^V_q, u_{q,m}^{V,s} \right\rangle = \langle u_{q,m}^{V,s}, \sin(k\cdot) u_{q,m}^{V,s}\rangle_{L^2_{\rm per}}.
\]

In the rest of the article, for all $p\in\N^*$, we will denote by $\nabla \mathcal{J}^s(V)|_{Y^p}$ the $2p+1$-dimensional real-valued vector so that
$$
\nabla \mathcal{J}^s(V)\big|_{Y^p} = \left( \partial_{d_p^V} \cJ^s(V), \partial_{d_{p-1}^V} \cJ^s(V), \cdots, \partial_{d_1^V} \cJ^s(V), \partial_{c_0^V} \cJ^s(V), \partial_{c_1^V} \cJ^s(V), \cdots, \partial_{c_p^V} \cJ^s(V) \right).
$$

\medskip

In order for the reader to better compare our adaptive algorithm with this naive one, we provide its pseudo-code below (Algorithm~\ref{algo:descent}).

\begin{algorithm}[H]
\textbf{\\Input:} \\
$p,s \in \N^*$\;
$W_0 \in Y_p$ : initial guess\;
$\varepsilon > 0 $: prescribed global precision\;
$\nu > 0$: tolerance for the norm of the gradient\;
\textbf{\\Output:} \\
$W_* \in Y_p$ such that $\| \nabla \cJ^s(W_*)\big|_{Y_p} \| \leq \nu $\;
\textbf{\\ Instructions:} \\
$n = 0$, $W = W_0$\;
\While{ $\| \nabla \cJ^{s}(W)\big|_{Y_p}| \| > \nu $}{
compute a descent direction  $D \in Y_p$ at $\cJ^{s}(W)$ (using \textbf{SD} / \textbf{PR} / \textbf{BFGS})\;
choose $t\in \R$ so that $\displaystyle t \in \mathop{\rm argmin}_{\overline{t}\in\R} \cJ^s(W + \overline{t}D)$\;
set $W \leftarrow W + t D$\;
}
return $W_* = W$.  
\caption{Naive optimisation algorithm}
\label{algo:descent}
\end{algorithm}

Although this method gives satisfactory numerical optimisers as shown in Section~\ref{sec:numerical_results}, its computational time grows very quickly with the discretisation parameters $p$ and $s$.
Besides, it is not clear how these parameters should be chosen a priori, given some target bands. This motivates the design of an adaptive algorithm.

\subsubsection{Adaptive algorithm}

In order to improve on the efficiency of the numerical optimisation procedure, we propose an adaptive algorithm, where the discretisation parameters $s$ or $p$ are increased during the optimisation process. To describe this procedure, we introduce two criteria to determine whether $s$ or $p$ need to be increased during the algorithm.

\medskip

As the parameter $s$ is increased, the approximated eigenvalues $\varepsilon_{q,m}^{V,s}$ becomes more accurate, and the discretised cost functional $\cJ^s$ gets closer to the true one $\cJ$. Our criterion for $s$ relies on the use of an a posteriori error estimator for the eigenvalue problem (\ref{eq:eigdisc}). More precisely, assume we can calculate at low numerical cost an estimator $\Delta_{q,m}^{V,s}\in \R_+$ such that
$$
|\varepsilon^V_{m,q} - \varepsilon_{m,q}^{V,s}| \leq \Delta_{q,m}^{V,s},
$$
(see Appendix~\ref{sec:appendixA}), then we would have that
\begin{align*}
| \mathcal{J}(V) - \mathcal{J}^s(V)| &  = \left|  \frac{1}{Q} \sum_{q\in\Gamma_Q^*} \sum_{m=1}^M \left( | b_m(q) - \varepsilon_{q,m}^V|^2 - | b_m(q) - \varepsilon_{q,m}^{V,s}|^2 \right) \right| \\
& = \left| \frac{1}{Q} \sum_{q\in\Gamma_Q^*} \sum_{m=1}^M \left( 2b_m(q) - \varepsilon_{q,m}^V - \varepsilon_{q,m}^{V,s}\right) \left( \varepsilon_{q,m}^{V,s} - \varepsilon_{q,m}^{V} \right) \right| \\
& \leq \frac{1}{Q} \sum_{q\in\Gamma_Q^*} \sum_{m=1}^M  \left( 2\left| b_m(q) - \varepsilon_{q,m}^{V,s}\right| + \Delta_{q,m}^{V,s}\right)\Delta_{q,m}^{V,s} =: \cS_V^s.
\end{align*}
The quantity $\cS_V^s$ estimates the error between $\mathcal{J}(V)$ and $\mathcal{J}^s(V)$ and therefore gives information on the necessity to adapt the value of the discretisation parameter $s$. 

\medskip

We now derive a criterion for the parameter $p$. When this parameter is increased, the minimisation space $Y_p$ gets larger. A natural way to decide whether or not to increase $p$ is therefore to consider the gradient of $\cJ^s$,
at the current minimisation point $W \in Y_p$, but calculated on a larger subspace $Y_{p'} \supset Y_p$ with $p' > p$.

\medskip

In practice, the natural choice $p' = p+1$ is inefficient. This is not a surprise, as there is no reason a priori to expect a sudden change at exactly the next Fourier mode. We therefore took the heuristic choice $p' = 2p$. More specifically, we define
\[
\cP^p_V := \left\| \nabla_V \cJ^s(V) \big|_{Y_{2p}} \right\|.
\]
Note that this estimator needs to be computed only when $V$ is a local minimum of $\cJ^s$ on $Y_p$. When this estimator is larger than some threshold, we increase $p$ so that the new space $Y_p$ contains the Fourier mode which provides the highest contribution in $ \left( \nabla_V \cJ^s(V) \right) \big|_{Y_{2p}}$.


\medskip

The adaptive procedure we propose is described in details in Algorithm~\ref{algo:adaptive}: 

\begin{algorithm}[H]

\textbf{\\Input:} \\
$p_0, s_0 \in \N^*$ : initial discretisation parameters\; 
$W_0 \in Y_{p_0}$ : initial guess\;
$\eta > 0 $: global discretisation precision\;
$\nu > 0 $: gradient norm precision\;

\textbf{\\Output:} \\
$p \geq p_0$, $s\geq s_0$ :  final discretisation parameters\;  
$W_* \in Y_p$ such that $\| \nabla \cJ^{s}(W_*)\big|_{Y_p} \|  \leq \nu$, $\cS^s_{W_*} \leq \eta$ and $\cP^p_{W_*} \leq \eta$\;

\textbf{\\ Instructions:} \\
$n = 0$, $W = W_0$\;
\While{$\| \nabla \cJ^{s}(W)\big|_{Y_p} \| > \nu$ or $\cS_{W}^{s} > \eta$ or $\cP^p_{W} > \eta$}{

\While{$\| \nabla \cJ_{p}^{s}(W)\big|_{Y_p} \| > \nu$}{
compute a descent direction  $D\in Y_{p}$ at $\cJ^{s}(W)$ (using \textbf{SD} / \textbf{PR} / \textbf{BFGS})\;
choose $t\in \R$ so that $\displaystyle t \in \argmin_{\overline{t}\in\R} \cJ^{s}(W + \overline{t}D)$\;
set $W \leftarrow W + t D$\;
}

\If{$\cS_{W}^{s} > \eta$ }{
set $s \leftarrow s + 1$\; 
}
\ElseIf{$\cP_{W}^{p} > \eta$}{
set $\displaystyle p \leftarrow \argmax_{p<\overline{p} \leq 2p} \quad \max \left( \left| \partial_{d_{\overline{p}}^V}\cJ^{s}(W)\right| , \left| \partial_{c_{\overline{p}}^V}\cJ^{s}(W)\right|\right)$\;
}
}
return $W_*= W$.
\caption{Adaptive optimisation algorithm}
\label{algo:adaptive}
\end{algorithm}

\subsection{Numerical results} 
\label{sec:numerical_results}

In this section, we illustrate the different algorithms presented above. 

\medskip

We consider the case where the target functions come from a target potential $V_{\rt} \in Y_{p_{\rt}}$, whose Fourier coefficients are randomly chosen for some $p_t\in \N^*$. 
We therefore take $b_m(q) := \varepsilon_{q,m}^{V_{\rt}, s_{\rt}}$, and try to recover the first $M$ functions $b_m$. The numerical parameters are $M = 3$, $Q = 25$, $\nu = 10^{-5}$, $\eta = 10^{-6}$ and $s_\rt = 20$. 
The initial guess is $W_0 = 0$. The naive algorithms are run with $s = s_\rt$ and $p=p_\rt$, while the adaptive algorithms start with $s_0 = p_0 = 1$. In addition, the a posteriori estimator is obtained with $s_{\rm ref}=250$ 
and $\theta = 0.01$ (see Appendix~\ref{sec:appendixA}). All tests are done with the naive and adaptive algorihms, with steepest descent (\textbf{SD}), conjugate gradient with Polak Ribiere formula (\textbf{PR}) 
and quasi Newton with the Broyden-Fletcher-Goldfarb-Shanno formula (\textbf{BFGS}).

\medskip

In our first test, we try to recover a simple shifted cosine function (i.e. $p_\rt = 1$). Results are shown in Figure~\ref{fig:1D_simple}. We observe that the bands and the potential are well reconstructed. 
We also notice that the adaptive algorithm takes more iterations to converge. However, as we will see later, most iterations are performed for low values of the parameters $s$ and $p$, and therefore are 
usually faster in terms of CPU time (see Table~\ref{tab:results_1D} below).
\begin{figure}[ht]
\centering 
\captionsetup{justification=centering}
\subfloat[Potentials]{\includegraphics[height=6cm, width=8cm]{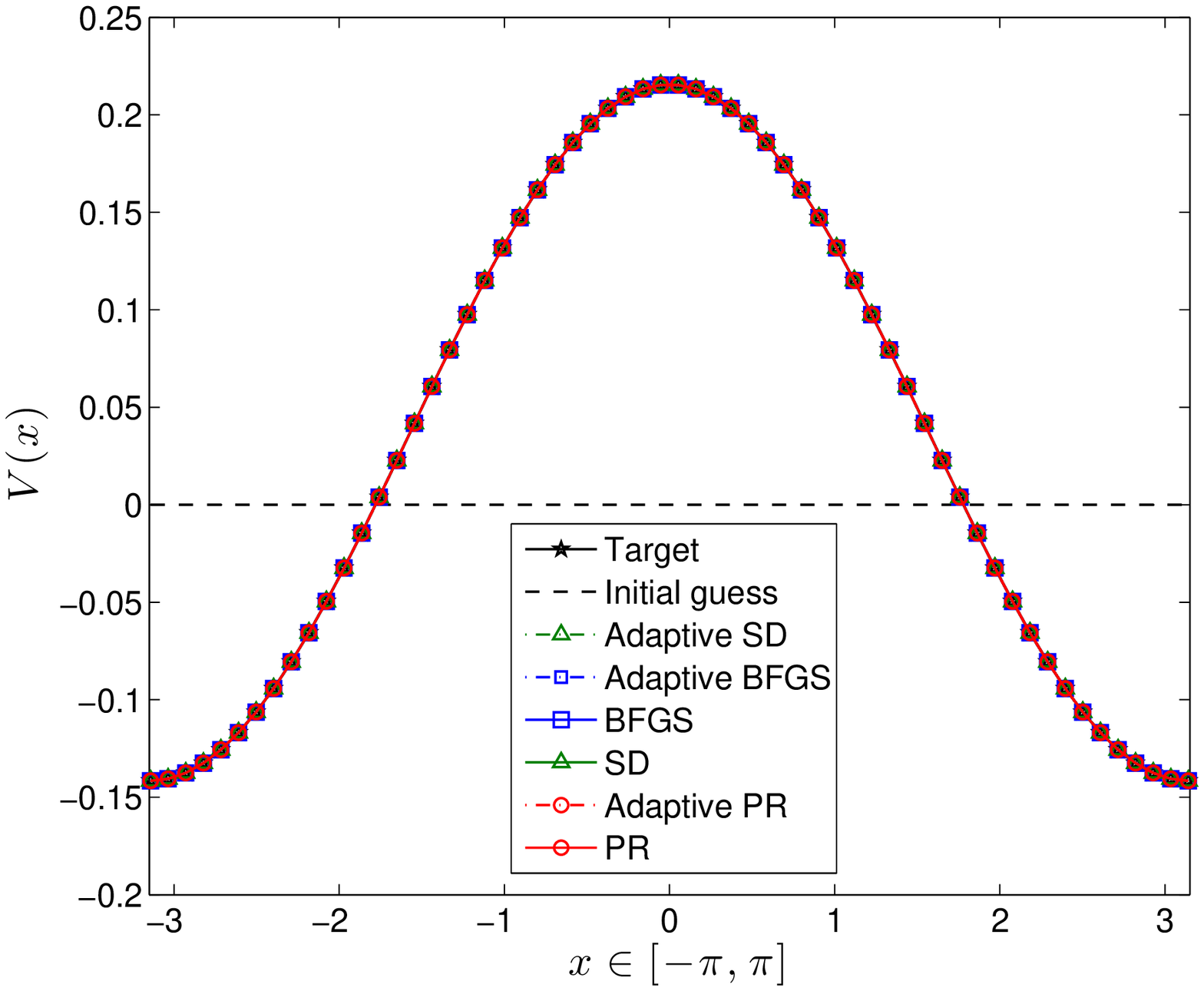}}
\subfloat[Bands]{\includegraphics[height=6cm, width=8cm]{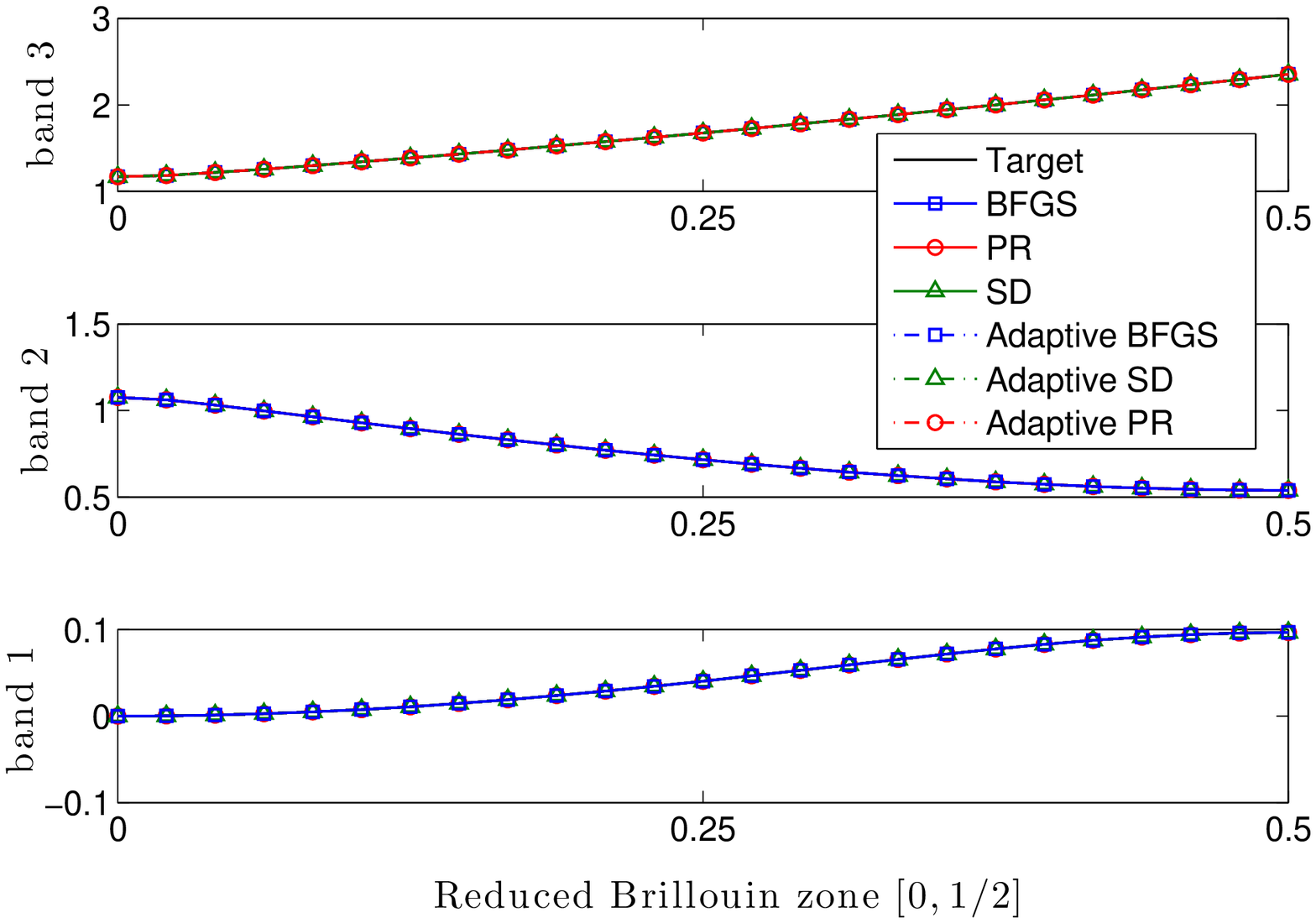}}\\ 
\subfloat[Evolution of $s$]{\includegraphics[height=6cm, width=8cm]{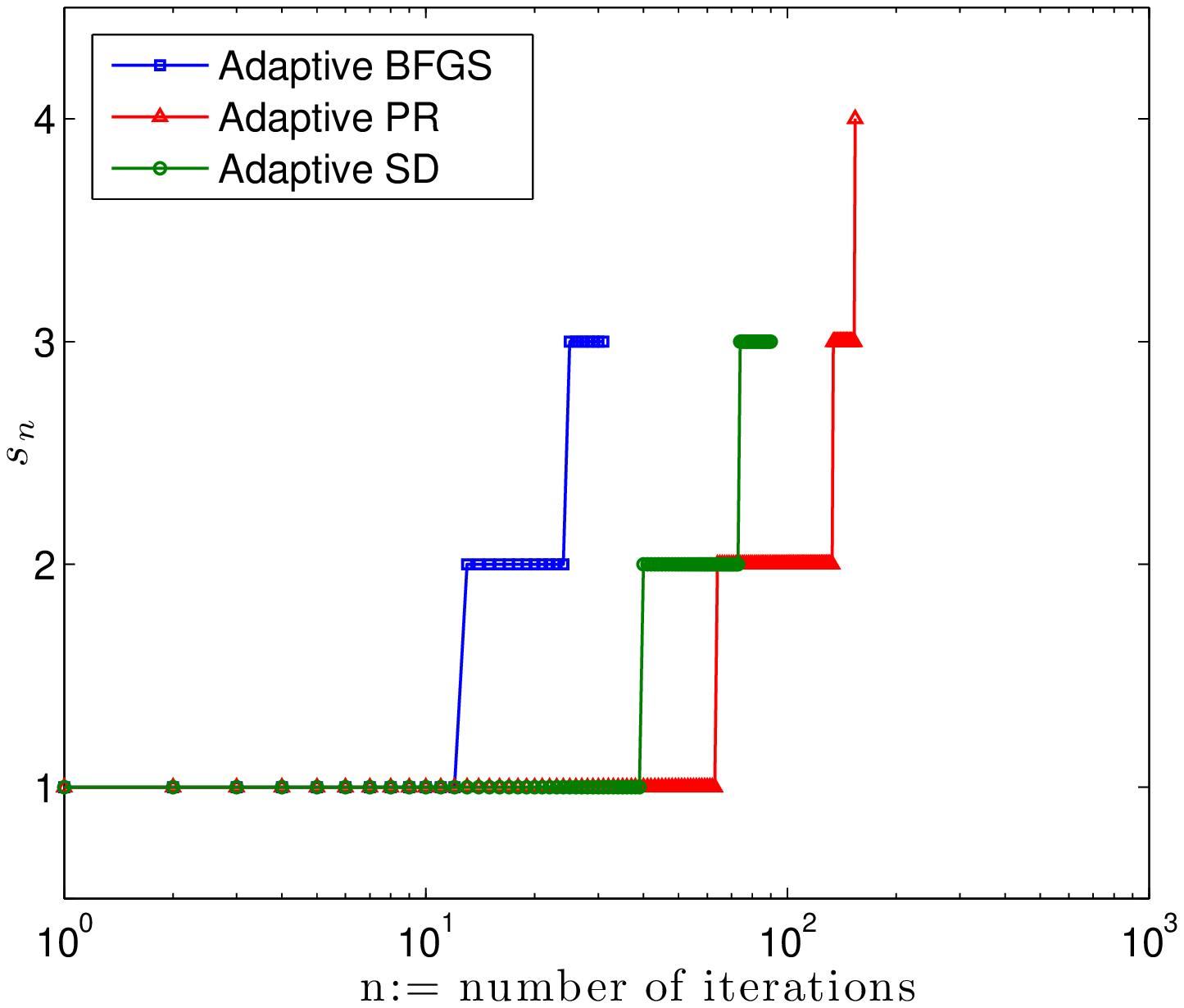}}
\subfloat[Evolution of $p$]{\includegraphics[height=6cm, width=8cm]{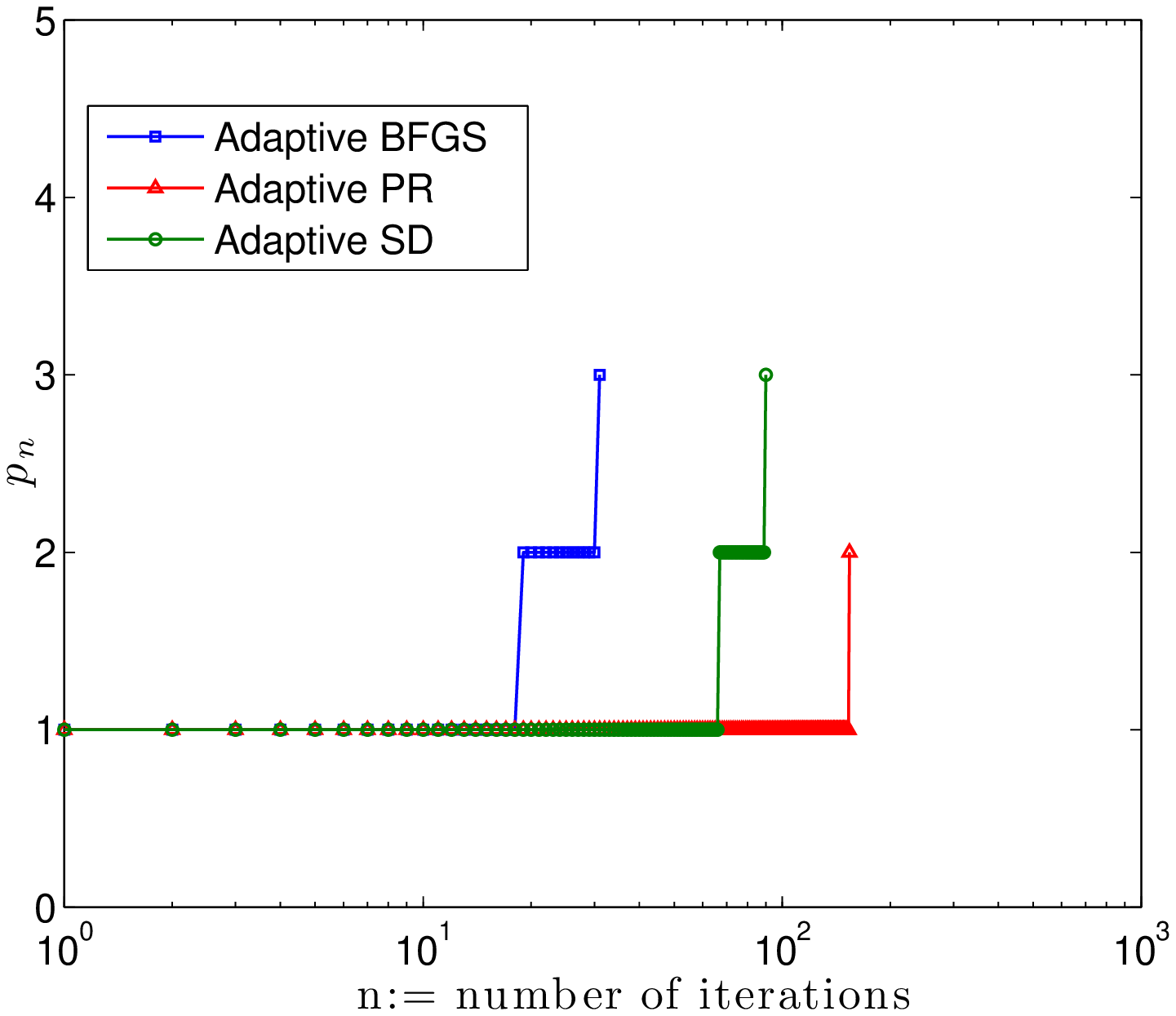}} \\
\subfloat[Convergence of the algorithms]{\includegraphics[height=6cm, width=8cm]{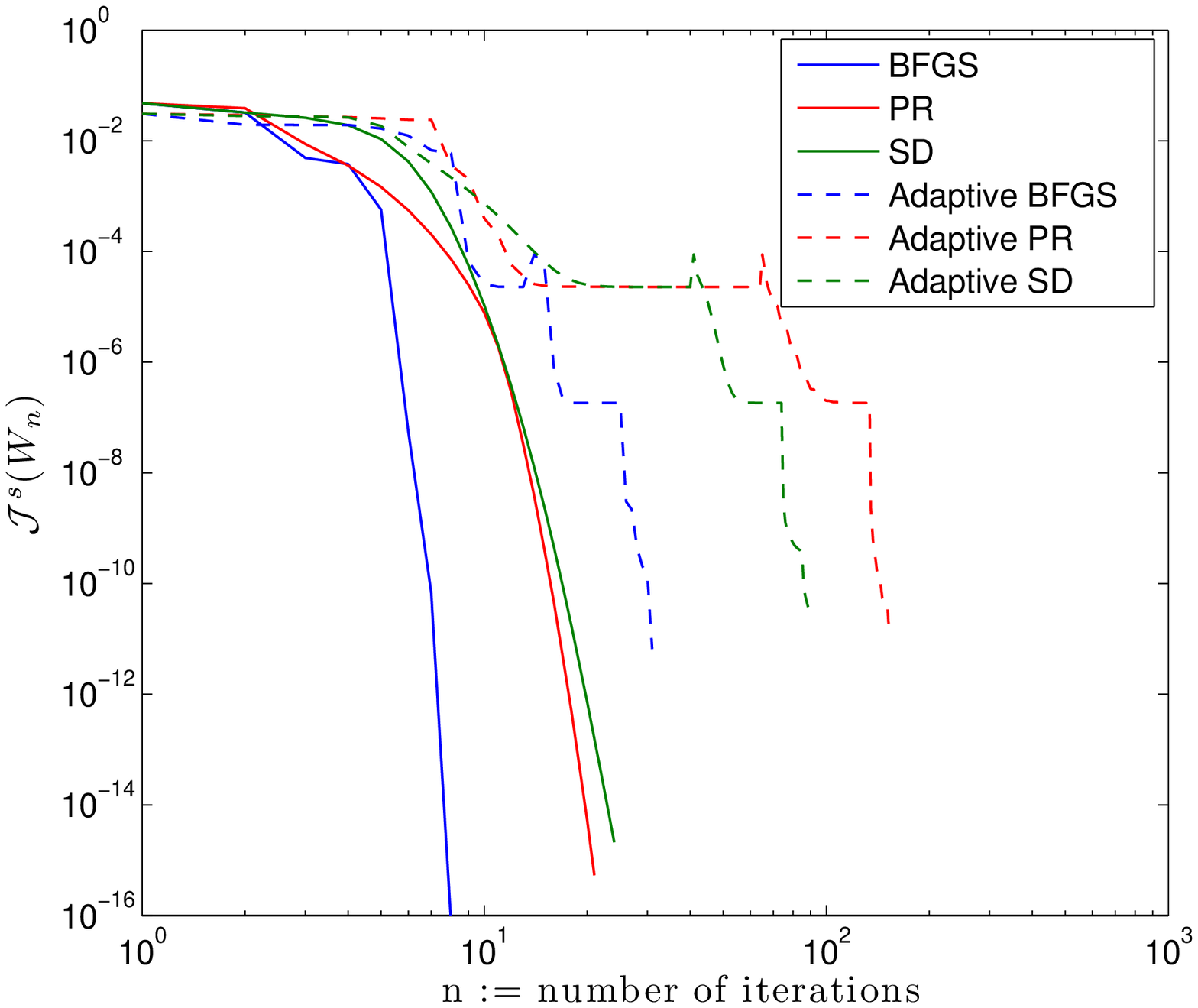}}
\caption{Recovery of the cosine potential}
\label{fig:1D_simple}
\end{figure}

In the second test case, we try to recover a more complex potential with $p_\rt = 8$ (see Figure~\ref{fig:1D_multimodal}). In this case, all the algorithms reproduce well the first bands, but fail to recover the potential. Actually, we see how different methods can lead to different local minima for the functional $\cJ$. This reflects the complex landscape of this function.
\begin{figure}[ht]
\centering 
\captionsetup{justification=centering}
\subfloat[Potentials]{\includegraphics[height=6cm, width=8cm]{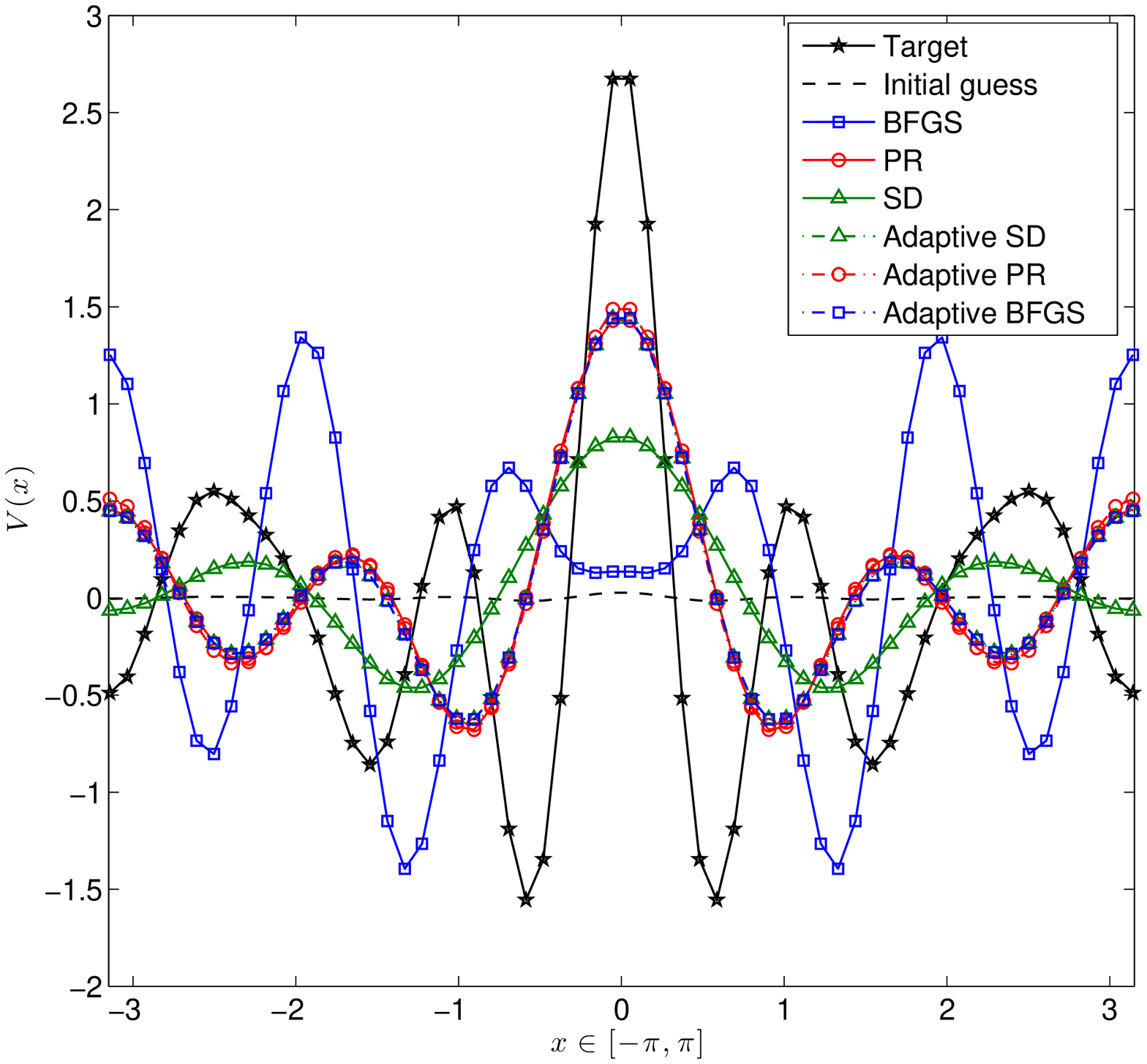}}
\subfloat[Bands]{\includegraphics[height=6cm, width=8cm]{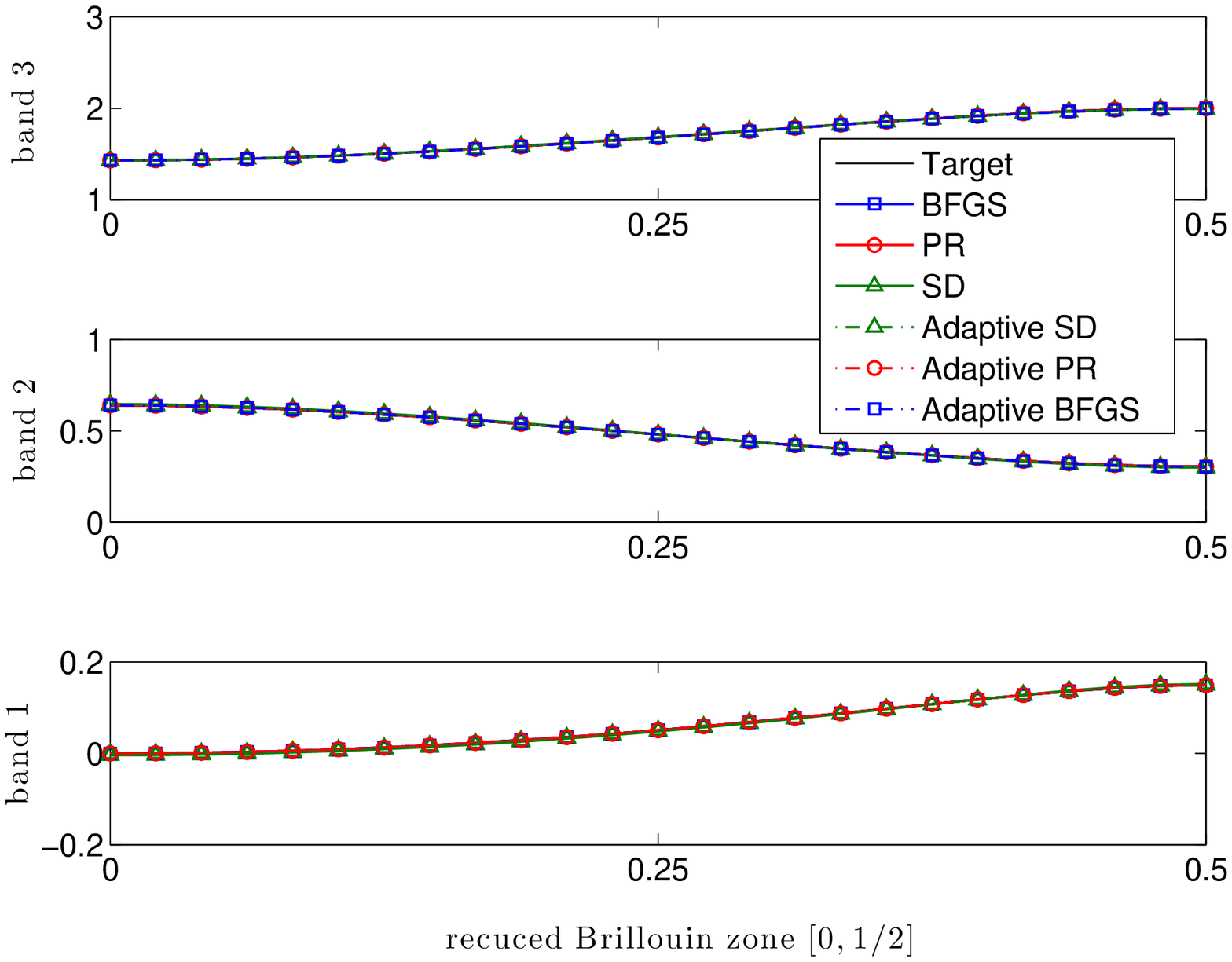}} \\
\subfloat[Evolution of $s$]{\includegraphics[height=6cm, width=8cm]{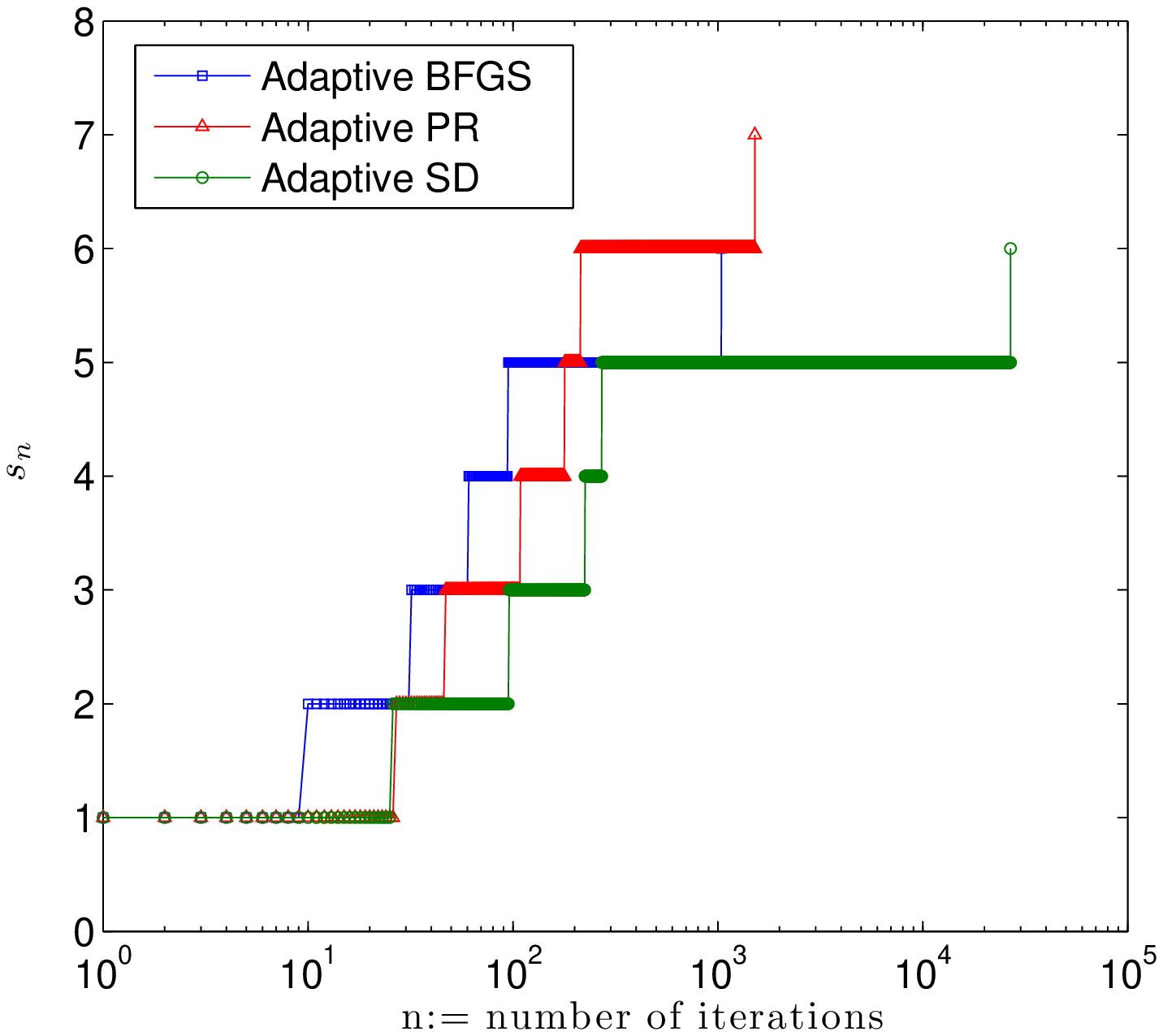}}
\subfloat[Evolution of $p$]{\includegraphics[height=6cm, width=8cm]{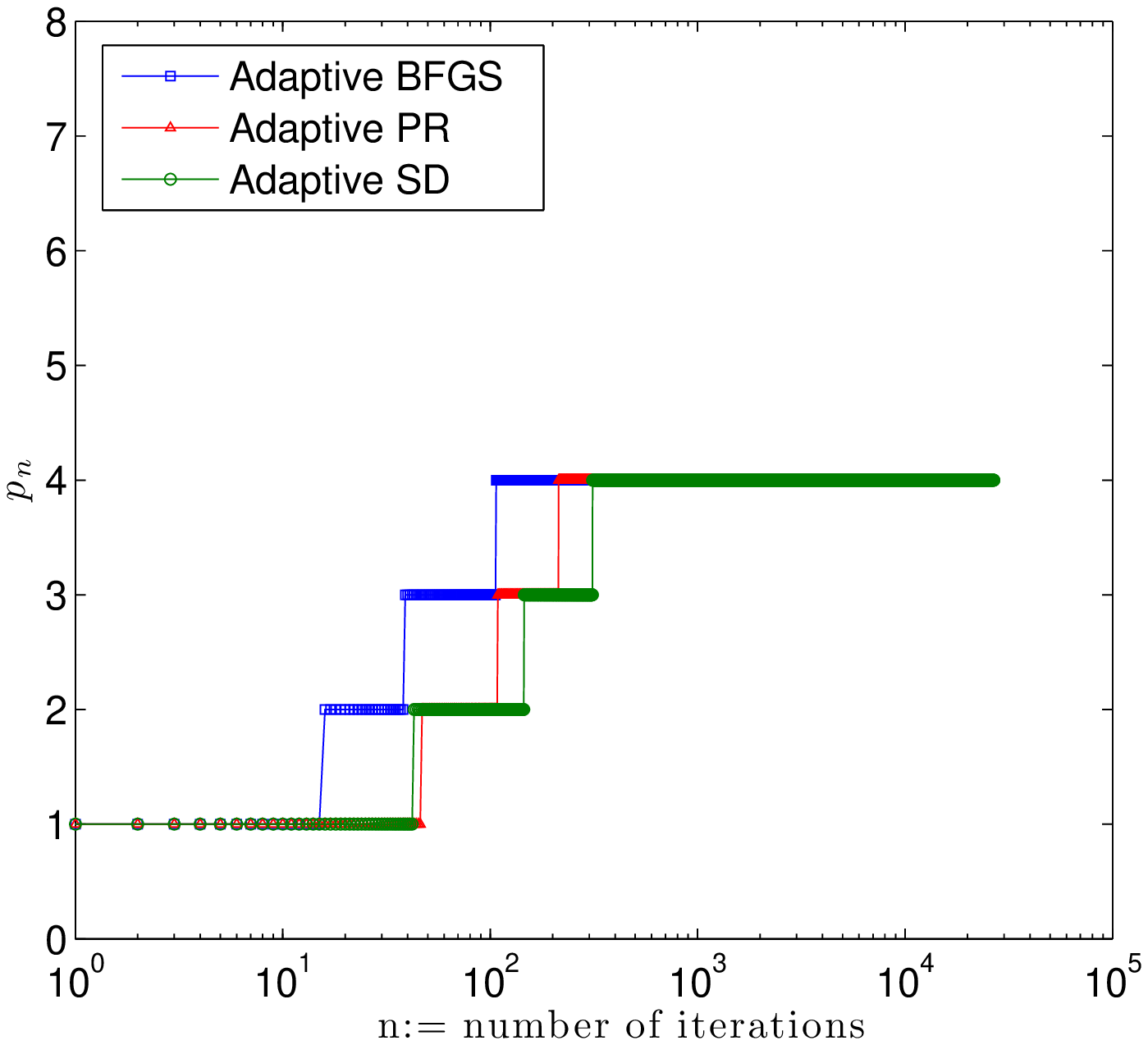}} \\
\subfloat[Convergence of the algorithms]{\includegraphics[height=6cm, width=8cm]{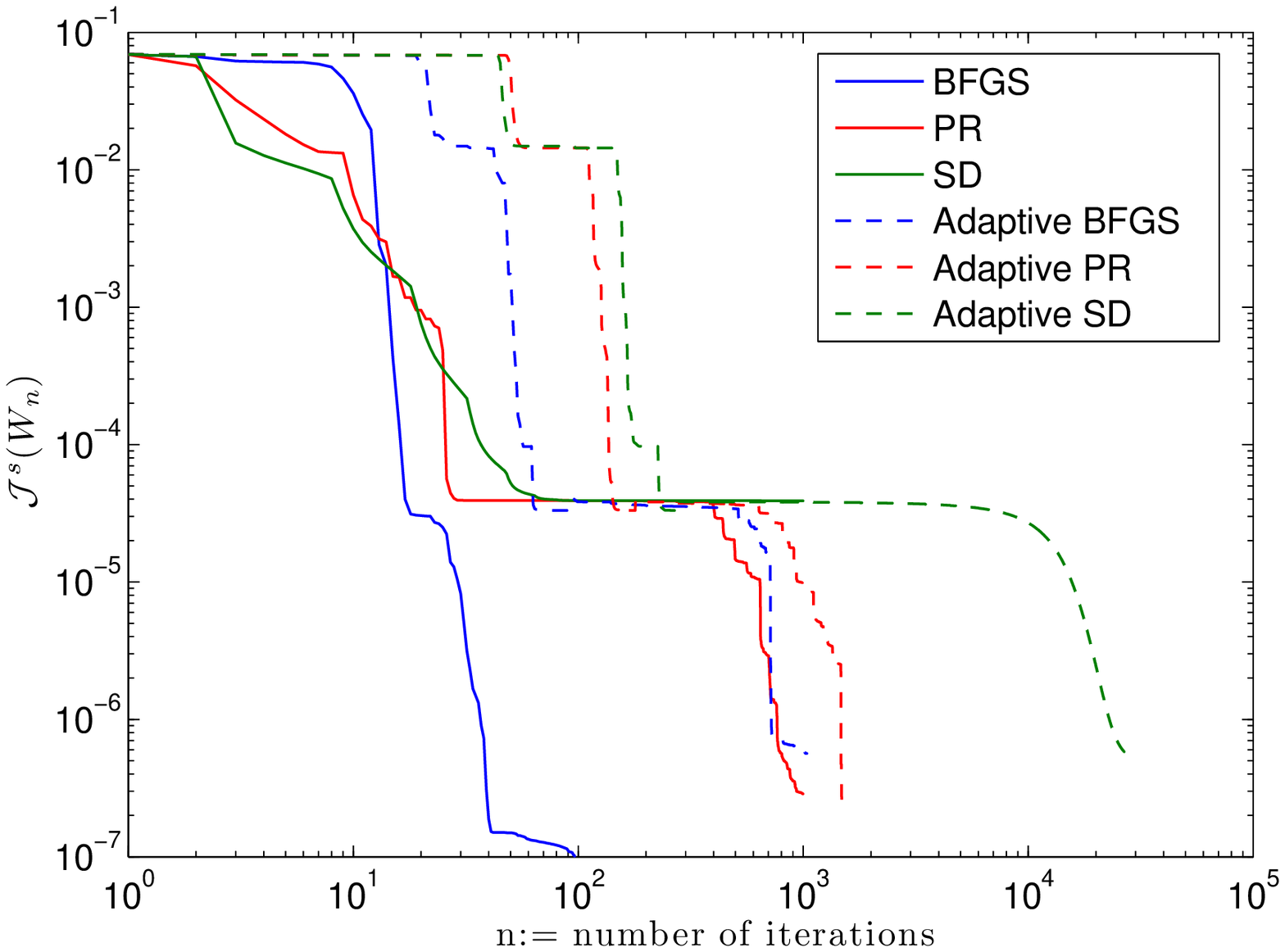}}
\caption{Recovery of an oscillating potential.}
\label{fig:1D_multimodal}
\end{figure}

We end this section by reporting results obtained with the different algorithms, and for different target potential $V_\rt \in Y_{p_\rt}$ with $p_\rt=1,4,8,12$ (see Table~\ref{tab:results_1D}). In this table, $N$ denotes the number of iterations, $s_N$ and $p_N$ are the values of the parameters $s$ and $p$ at the last iteration (in particular, for the naive algorithms, we have $s_N = s_\rt = 20$ and $p_N = p_\rt$). Lastly, for each algorithm \textbf{algo}, we define a relative CPU time 
\[
 \tau_{\rm \textbf{algo}} = \dfrac{t_{\rm \textbf{algo}}}{t_{\rm \textbf{SD}}},
\]
where $ t_{\rm \textbf{algo}}$ is the CPU time consumed by the algorithm \textbf{algo} and $t_{\rm \textbf{SD}}$ is the CPU time consumed by the classical steepest descent. 
In particular, $\tau_{\rm \textbf{SD}}=1$.
\begin{table}[H]
\centering
\begin{tabular}{|c|c|c|c|c|c|c|c|}
\hline 
\multirow{2}{*}{$p_\rt$}		    & -	&\multicolumn{2}{c|}{\textbf{BFGS}}  	&\multicolumn{2}{c|}{\textbf{PR}} 	&\multicolumn{2}{c|}{\textbf{SD}} \\  \cline{2-8}
& - 		     &\textbf{naive} 				&\textbf{adaptive}  &\textbf{naive}  			&\textbf{adaptive} 	 &\textbf{naive}  		    &\textbf{adaptive}  \\  \hline \hline
\multirow{2}{*}{$1$}
&$\tau$ 			&0.259	&1.176	&0.929	&1.320	&\textcolor{red}{1}	&1.255 	\\  \cline{2-8}
&$N$ 			&8		&31		&21		&154		&24		&90 		\\  \cline{2-8}
&$s_N$  			&\textcolor{red}{20}		&3 		& \textcolor{red}{20} 		&4		&\textcolor{red}{20}		&3		\\  \cline{2-8}
&$p_N$  			&\textcolor{red}{1}		&3 		& \textcolor{red}{1} 		&2		&\textcolor{red}{1}		&3		\\  \hline \hline 

%


\multirow{2}{*}{$4$}
&$\tau$ 	&0.070	&0.009	&0.464	&0.281	&\textcolor{red}{1}	&0.259 \\  \cline{2-8}
&$N$	&54		&1424	&1927	&7091	&8453	&19095 \\  \cline{2-8}
&$s_N$	&\textcolor{red}{20}		&8		&\textcolor{red}{20}		&7		&\textcolor{red}{20}		&5		\\  \cline{2-8}
&$p_N$  	&\textcolor{red}{4}		&5 		& \textcolor{red}{4}		&3		&\textcolor{red}{4}		&3		\\  \hline \hline 

\multirow{2}{*}{$8$}
&$\tau$ 		&0.470	&0.151	&1.090	&0.144	&\textcolor{red}{1}	&0.519 \\ \cline{2-8}
&$N$ 		&553		&1041	&1023	&1515	&7326	&26783 \\ \cline{2-8}
&$s_N$ 		&\textcolor{red}{20}		&6		&\textcolor{red}{20}		&7		&\textcolor{red}{20}		&6        \\ \cline{2-8}
&$p_N$		&\textcolor{red}{8}		&4		&\textcolor{red}{8}		&4		&\textcolor{red}{8}		&4    \\ \hline \hline


\multirow{2}{*}{$12$}
&$\tau$	&0.007	&0.001	&0.054	&0.004	&\textcolor{red}{1}	&0.044  \\ \cline{2-8}
&$N$ 	&765		&2474	&2413	&2727	&50312	&34865 \\ \cline{2-8}
&$s_N$	&\textcolor{red}{20}		&9		&\textcolor{red}{20}		&9		&\textcolor{red}{20}		&9	     \\ \cline{2-8}
&$p_n$	&\textcolor{red}{12}		&8		&\textcolor{red}{12}		&8		&\textcolor{red}{12}		&8      \\ \hline 

%
\end{tabular}
\caption{Results for recovery test with different algorithms. Red values are reference values.}
\label{tab:results_1D}
\end{table}

We notice that although the adaptive approach requires more iterations to converge, it is usually faster than the naive one. As we already mentioned, this is due to the fact that most of
the iterations are performed with small values of $p$ and $s$, and are therefore faster. Moreover, we notice that the adaptive algorithms tend to find an optimised 
potential which $p_N \leq p_\rt$, i.e. a less oscillatory potential than the target one.

 \section*{Ackowledgements}
 The authors heartily thank \'Eric Canc\`es, Julien Vidal, Damiano Lombardi and Antoine Levitt for their great help in this work and for inspiring discussions. 
 The IRDEP institute is acknowledged for funding.

 \appendix
 \section{A posteriori error estimator for the eigenvalue problem}
 \label{sec:appendixA}
 We present in this appendix the a posteriori error estimator for eigenvalue problems that we use in Section~\ref{sec:numerical_results}. More details about this estimator are given in~\cite{a_posteriori}.

 \medskip

 Let $\cH$ be a finite dimensional space of size $N_{\rm ref}$ and let $A$ be a self-adjoint operator on $\cH$. In our case, $\cH$ is some $X_{s_{\rm ref}}$ for some 
 large $s_{\rm ref} \gg 1$, and $A = A^{V,s_{\rm ref}}_q$. The eigenvalues of $A$, counting multiplicities are denoted by $\varepsilon_1 \le \varepsilon_2 \le \cdots \le \varepsilon_{N_{\rm ref}}$.

 \medskip

 For $N \ll N_{\rm ref}$, we consider $X_N$ a finite dimensional subspace of $\cH$. We denote by $\Pi_{X_N}$ the orthogonal projection on $X_N$, and by $A^N := \Pi_{X_N} A \Pi_{X_N}^*$. 
 The eigenvalues of $A^N$ are denoted by $\varepsilon_1^N \le \varepsilon_2^N \le \cdots \le \varepsilon_{N}^N$. Let us also denote by $\left( u_m^N \right)_{1 \le j \le N}$ a corresponding orthogonal basis of $X^N$, so that
 \[
 \forall 1 \le m \le N, \quad A^N u_m^N = \varepsilon_m^N u_m^N.
 \]
 We recall that, from the min-max principle, it holds that $\varepsilon_{m} \le \varepsilon_{m}^N$. A certified a posteriori error estimator for the $m$-th eigenvalue is a non-negative real number $\Delta_m^N \in \R_+$ such that 
 $$
 \varepsilon_{m}^N - \varepsilon_{m} \leq \Delta_{m}^N.
 $$
 We also require that the expression of $\Delta_m^N$ only involves the approximated eigenpair $\varepsilon_{m}^N$ and $u_{m}^N$ (and not $\varepsilon_m$). 



 \begin{proposition}\label{prop:posteriori}
 Assume that $\varepsilon_m$ (resp. $\varepsilon_m^N$) is a non-degenerate eigenvalue of $A$ (resp. $A^N$), and that 
 \begin{equation}\label{eq:assposter}
 0 < \varepsilon_m^N - \varepsilon_m < \mbox{\rm dist} \left(\varepsilon_m^N, \sigma(A)\setminus\left\{\varepsilon_m\right\} \right).
 \end{equation}
 Let $\lambda_m < \varepsilon_m$. Then there exists $\delta_m > 0$ such that, for all $0 \le \delta < \delta_m$, we have
 \begin{equation}\label{eq:apost}
 \varepsilon_m^N - \varepsilon_m \leq \left\langle r_m^N,  \left(A -c_\delta\right)^{-1} \left(A -d_\delta\right)\left(A -c_\delta\right)^{-1} r_m^N\right\rangle,
 \end{equation}
 where we set $c_\delta :=  \varepsilon_m^N + \delta $, $d_\delta:= \lambda_m + \delta$, and where $r_m^N:= \left( A  - \varepsilon_m^N \right) u_m^N$ is the residual.
 \end{proposition}

 \begin{proof}
 Assumption (\ref{eq:assposter}) implies that $\varepsilon_m^N \notin \sigma\left( A \right)$, so that $\left( A  - \varepsilon_m^N \right)$ is invertible. From the fact that $\langle u_m^N , A u_m^N \rangle = \varepsilon_m^N$, and the definition of the residual, it holds that
 \begin{equation}\label{eq:firsteq}
 \varepsilon_m^N - \varepsilon_m = \left\langle r_m^N , \left(A -\varepsilon_m^N\right)^{-1} \left(A - \varepsilon_m \right)\left(A -\varepsilon_m^N\right)^{-1} r_m^N \right\rangle.
 \end{equation}
 Thus, a sufficient condition for (\ref{eq:apost}) to hold is that 
 $$
 \left(A -c_\delta\right)^{-1} \left(A -d_\delta\right)\left(A -c_\delta\right)^{-1} \geq \left(A -\varepsilon_m^N\right)^{-1} \left(A - \varepsilon_m \right)\left(A -\varepsilon_m^N\right)^{-1}.
 $$
 Thanks to the spectral decomposition of $A$, this is the case if and only if, 
 $$
 \forall 1 \le \widetilde{m} \le N_{\rm ref}, \quad \frac{\varepsilon_{\widetilde{m}} - d_\delta}{\left( \varepsilon_{\widetilde{m}} - c_\delta\right)^2} \geq \frac{\varepsilon_{\widetilde{m}} - \varepsilon_m}{\left( \varepsilon_{\widetilde{m}} - \varepsilon_m^N\right)^2}. 
 $$
 Denoting by $\eta:= \mbox{\rm dist} \left(\varepsilon_m^N, \sigma(A)\setminus\left\{\varepsilon_m\right\}\right) - \left(\varepsilon_m^N - \varepsilon_m\right)$, 
 this holds true as soon as $\delta \leq \delta_m:= \min\left(\varepsilon_m - \lambda_m, \eta\right)$. The result follows.
 \end{proof}

 In order to use the left-side of~\eqref{eq:apost} as an a posteriori estimator, we need to choose $\lambda_m < \varepsilon_m$ and $\delta_m > 0$. For the choice of $\lambda_m$, we follow~\cite{Wolkowitz}, and notice that 
 $$
 \varepsilon_m \geq \lambda_m:= \mu - \left( \frac{N_{\rm ref} -m -1}{m+1} \right)^{1/2} \sigma,
 $$
 where we set
 $$
 \mu :=\frac{1}{N_{\rm ref}} \Tr A \quad \text{and} \quad \sigma^2 := \frac{1}{N_{\rm ref}} \Tr A^2 - \mu^2. 
 $$
 For the choice of $\delta_m$, we chose the simple rule
 $$
 \delta_m  = \theta \left(\varepsilon_m^N -\kappa\right) \quad \text{with} \quad 0 < \theta \ll 1 \quad \text{ and } \kappa \in \R \quad \text{independent of $m$}.
 $$
 The real number $\kappa$ is chosen to be an a priori lower bound of the lowest eigenvalue $\varepsilon_1$ of $A$.
 This choice is heuristic in the sense that we cannot guarantee that the assumptions of Proposition~\ref{prop:posteriori} are satisfied. However, the encouraging 
 numerical results we obtain below motivated our choice to use such an estimator (see Section~\ref{sec:posternum}).

 \paragraph{Numerical test}\label{sec:posternum}

 To illustrate the efficiency of our heuristic, we tested it to compute the first bands of the Hill's operator $A^V$ with
 $$
 V(x) =  \sum \limits_{k=-3}^3 \hat{V}_k e_k, \quad \text{where} \quad
 \hat{V}_0 = 2 \quad \text{and} \quad \overline{\hat{V}_{-1}} =\overline{\hat{V}_{-2}} =  \hat{V}_1 = \hat{V}_2 = 1 + 0.5 \, i.
 $$
 The reference operator is $A := A_{q}^{V,s_{\rm ref}}$ with $s_{\rm ref} = 250$, and the first three bands are computed on the space $X^s$ defined in~\ref{def:X_s} with $s = 6$. 
 We plot in Figure~\ref{fig:posteriori_1D} the true error $\varepsilon_{q,m}^{V,s} - \varepsilon_{q,m}^{V,s_{\rm ref}}$ for $m = 1,2,3$, and the corresponding a posteriori error with $\kappa = 0$ and different values of $\theta$ (namely $\theta = 0.1, 0.5, 1$). We observe that our estimator is sharp for a large range of $\theta$.

 \iffigure
 \begin{figure}[ht]
 \centering
 \subfloat[Potential $V$]{\includegraphics[width=7cm, height = 6cm]{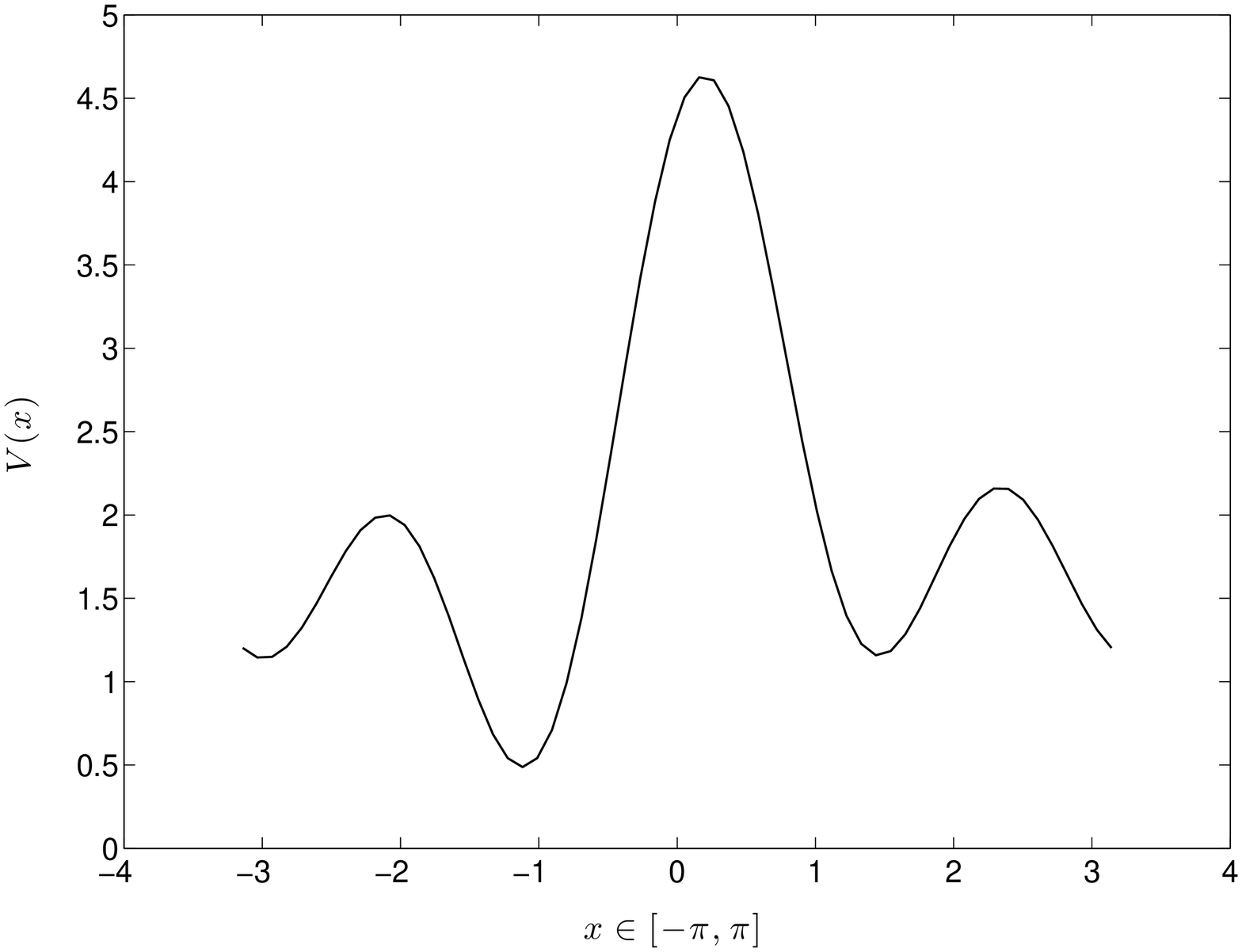}} 
 \subfloat[$m=1$]{\includegraphics[width=7cm,height = 6cm]{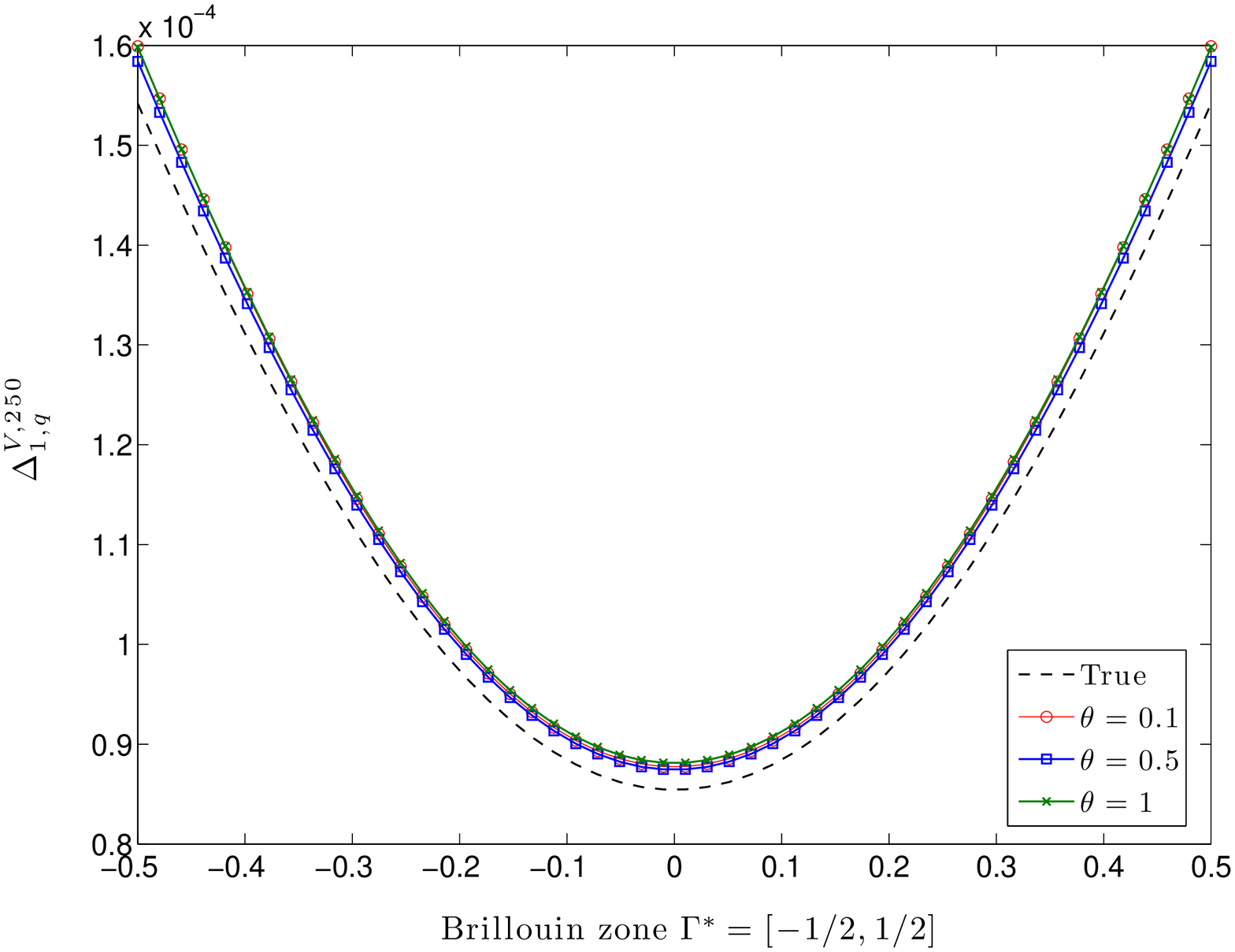}} \\
 \subfloat[$m=2$]{\includegraphics[width=7cm,height = 6cm]{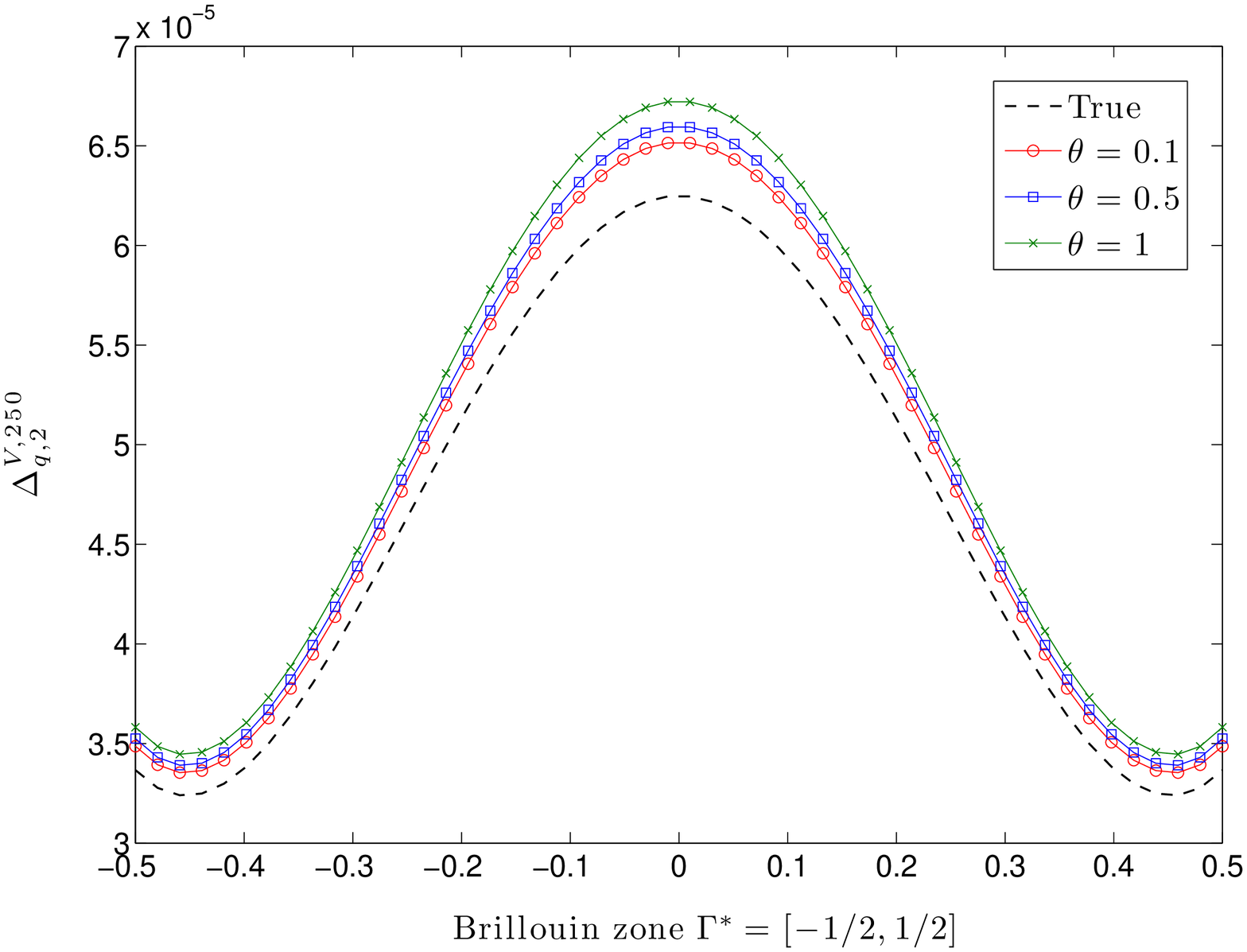}} 
 \subfloat[$m=3$]{\includegraphics[width=7cm,height = 6cm]{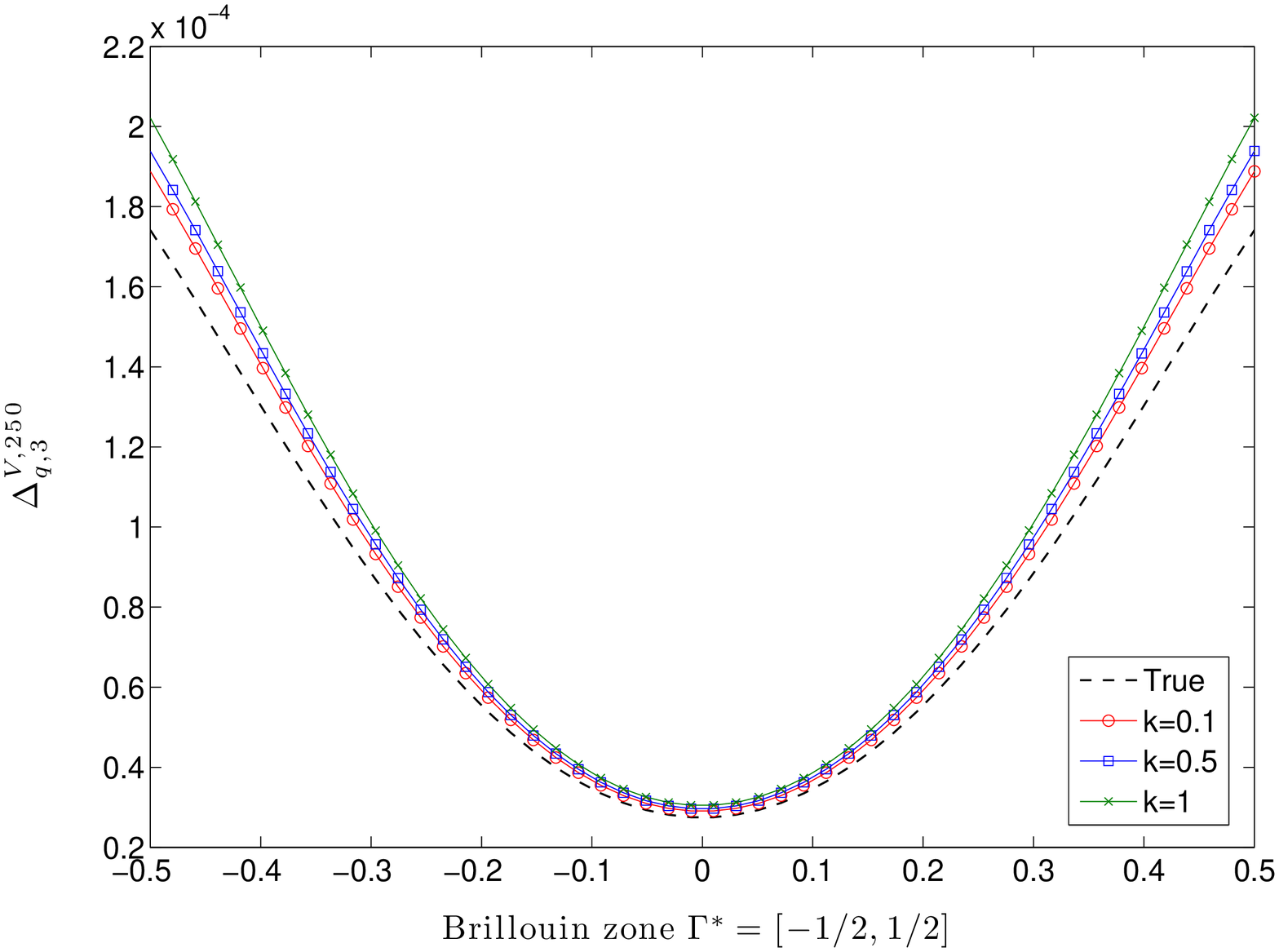}} 
 \caption{Numerical validation of the a posteriori error estimator proposed in Appendix~\ref{sec:appendixA}.}  
 \label{fig:posteriori_1D}
 \end{figure}
 \fi

 \bibliographystyle{plain}
 \bibliography{bands_biblio}

\begin{thebibliography}{10}

\bibitem{BakhtaPhD}
A.~Bakhta.
\newblock {\em Mathematical models and numerical simulation of photovoltaic
  materials}.
\newblock PhD thesis, Universit\'e Paris Est - Ecole des Ponts ParisTech, 2017.

\bibitem{a_posteriori}
A.~Bakhta and D.~Lombardi.
\newblock An a posteriori error estimator based on shifts for positive
  hermitian eigenvalue problems.
\newblock {\em https://hal.inria.fr/hal-01584180/}, 2017.

\bibitem{Bonnans}
J.F. Bonnans, J.Ch. Gilbert, C.~{Lemar{\'e}chal}, and C.~{Sagastiz{\'a}bal}.
\newblock {\em Numerical Optimization}.
\newblock Springer Verlag, 2003.

\bibitem{Dias2016}
N.C. Dias, C.~Jorge, and J.N. Prata.
\newblock One-dimensional {S}chr{\"o}dinger operators with singular potentials:
  A {S}chwartz distributional formulation.
\newblock {\em J. of Differential Equations}, 260(8):6548--6580, 2016.

\bibitem{Eskin}
G.~Eskin.
\newblock Inverse spectral problem for the {S}chr{\"o}dinger equation with
  periodic vector potential.
\newblock {\em Commun. Math. Phys}, 125(2):263--300, 1989.

\bibitem{EskinRalstonTrubowiz1}
J.~Eskin, J.~Ralston, and E.~Trubowiz.
\newblock On isospectral periodic potential in $\mathbb{R}^n$. {I}.
\newblock {\em Commun. Pure Appl. Maths.}, 37(5):647--676, 1984.

\bibitem{EskinRalstonTrubowiz2}
J.~Eskin, J.~Ralston, and E.~Trubowiz.
\newblock On isospectral periodic potential in $\mathbb{R}^n$. {II}.
\newblock {\em Commun. Pure Appl. Maths.}, 37(6):715--753, 1984.

\bibitem{Evans}
L.C. Evans.
\newblock {\em Partial Differential Equations}.
\newblock Graduate studies in mathematics. American Mathematical Society, 1998.

\bibitem{FreilingYurko}
G.~Freiling and V.~Yurko.
\newblock {\em Inverse {S}turm-{L}iouville problems and their applications}.
\newblock Nova Science Publishers, 2001.

\bibitem{gesztesy2006spectral}
F.~Gesztesy and M.~Zinchenko.
\newblock On spectral theory for {S}chr{\"o}dinger operators with strongly
  singular potentials.
\newblock {\em Math. Nachr.}, 279(9-10):1041--1082, 2006.

\bibitem{hryniv2001schr}
R.O. Hryniv and Y.V. Mykytyuk.
\newblock 1-{D} {S}chr\"odinger operators with periodic singular potentials.
\newblock {\em Methods Funct. Anal. Topology}, 7(4):31--42, 2001.

\bibitem{hryniv2003inverse}
R.O. Hryniv and Y.V. Mykytyuk.
\newblock Inverse spectral problems for {S}turm-{L}iouville operators with
  singular potentials.
\newblock {\em Inverse Problems}, 19(3):665, 2003.

\bibitem{hryniv2003inverse3}
R.O. Hryniv and Y.V. Mykytyuk.
\newblock Inverse spectral problems for {S}turm-{L}iouville operators with
  singular potentials. {III}. {R}econstruction by three spectra.
\newblock {\em J. Math. Anal. Appl.}, 284(2):626--646, 2003.

\bibitem{hryniv2004half}
R.O. Hryniv and Y.V. Mykytyuk.
\newblock Half-inverse spectral problems for {S}turm-{L}iouville operators with
  singular potentials.
\newblock {\em Inverse Problems}, 20(5):1423, 2004.

\bibitem{hryniv2004inverse2}
R.O. Hryniv and Y.V. Mykytyuk.
\newblock Inverse spectral problems for {S}turm-{L}iouville operators with
  singular potentials. {II}. {R}econstruction by two spectra.
\newblock {\em North-Holland Mathematics Studies}, 197:97--114, 2004.

\bibitem{hryniv2006inverse4}
R.O. Hryniv and Y.V. Mykytyuk.
\newblock Inverse spectral problems for {S}turm-{L}iouville operators with
  singular potentials. {IV}. {P}otentials in the {S}obolev space scale.
\newblock {\em Proceedings of the Edinburgh Mathematical Society},
  49(2):309--329, 2006.

\bibitem{kato1972schrodinger}
T.~Kato.
\newblock Schr{\"o}dinger operators with singular potentials.
\newblock {\em Israel Journal of Mathematics}, 13(1):135--148, 1972.

\bibitem{Kuchment2016}
P.~Kuchment.
\newblock An overview of periodic elliptic operators.
\newblock {\em Bull. Amer. Math. Soc.}, 53(3):343--414, 2016.

\bibitem{lieb2001analysis}
E.H. Lieb and M.~Loss.
\newblock {\em Analysis}, volume~14 of {\em Graduate studies in mathematics}.
\newblock 2001.

\bibitem{mikhailets2008}
V.~Mikhaelets and V.~Molyboga.
\newblock One-dimensional {S}chr\"odinger operators with singular periodic
  potentials.
\newblock {\em Methods Funct. Anal. Topology}, 14(2):184--200, 2008.

\bibitem{PoschelTrubowiz}
J.~P{\"o}schel and E.~Trubowiz.
\newblock {\em Inverse spectral theory. Pure and applied mathematics}.
\newblock Academic Press, 1987.

\bibitem{ReedSimon}
M.~Reed and B.~Simon.
\newblock {\em Methods of modern mathematical physics. {IV}: {A}nalysis of
  operators}.
\newblock Elsevier, 1978.

\bibitem{Trudinger1967harnack}
N.S. Trudinger.
\newblock On {H}arnack type inequalities and their application to quasilinear
  elliptic equations.
\newblock {\em Commun. Appl. Math.}, 20(4):721--747, 1967.

\bibitem{Veliev}
O.~Veliev.
\newblock {\em Multidimensional periodic {S}chr{\"o}dinger operator.
  {P}erturbation theory and applications.}
\newblock Academic Press, 2015.

\bibitem{Wolkowitz}
H.~Wolkowicz and G.P.H. Styan.
\newblock Bounds on eigenvalues using traces.
\newblock {\em Linear Algebra Appl.}, 29:471--506, 1980.

\end{thebibliography}

 \end{document}